\documentclass[11pt]{amsart}%
\usepackage{amssymb}
\usepackage{amsmath}
\usepackage{amsthm}
\usepackage{amscd,bm}
\usepackage{enumitem,graphicx,tikz, pst-node, pst-plot, pstricks, csquotes}
\usepackage[margin=1.0in]{geometry}%
\usepackage{cite,bbm,esint}

\DeclareMathOperator{\supp}{supp}

\DeclareMathOperator*{\diam}{diam}

\DeclareMathOperator{\diag}{diag}
\DeclareMathOperator{\lip}{Lip}
\DeclareMathOperator{\dist}{dist}

\theoremstyle{plain}
\newtheorem{theorem}{Theorem}[section]
\newtheorem{proposition}{Proposition}[section]
\newtheorem{corollary}{Corollary}[section]
\newtheorem{lemma}{Lemma}[section]

\theoremstyle{remark}
\newtheorem{remark}{Remark}[section]
\newtheorem{examples}{Examples}[section]
\newtheorem{assumption}{Assumption}[section]

\makeatletter
\@namedef{subjclassname@2020}{\textup{2020} Mathematics Subject Classification}
\makeatother

\begin{document}

\title{Non-local boundary energy forms for quasidiscs: Codimension gap and approximation}

\author{Simone Creo$^1$}
\address{$^1$Dipartimento di Scienze di Base e Applicate per l'Ingegneria,
Sapienza Universit\`{a} di Roma,
Via A. Scarpa 16, 00161 Roma, Italy}
\email{simone.creo@uniroma1.it}

\author{Michael Hinz$^2$}
\address{$^2$Fakult\"at f\"ur Mathematik, Universit\"at Bielefeld, Postfach 100131, 33501 Bielefeld, Germany}
\email{mhinz@math.uni-bielefeld.de}

\author{Maria Rosaria Lancia$^3$}

\address{$^3$Dipartimento di Scienze di Base e Applicate per l'Ingegneria,
Sapienza Universit\`{a} di Roma,
Via A. Scarpa 16, 00161 Roma, Italy}
\email{maria.lancia@sbai.uniroma1.it}

\begin{abstract}
We consider non-local energy forms of fractional Laplace type on quasicircles and prove that they can be approximated by similar energy forms on polygonal curves. The approximation is in terms of generalized Mosco convergence along a sequence of varying Hilbert spaces. The domains of the energy forms are the natural trace spaces, and we focus on the case of quasicircles of Hausdorff dimension greater than one. The jump in Hausdorff dimension results in a mismatch of fractional orders, which we compensate by a suitable choice of kernels. We provide approximations of quasidiscs by polygonal $(\varepsilon,\infty)$-domains with common parameter $\varepsilon>0$ and show convergence results for superpositions of Dirichlet integrals and non-local boundary energy forms. 
\tableofcontents
\end{abstract}


\maketitle

\section{Introduction}

We investigate planar domains with fractal boundaries and non-local boundary energy forms of fractional Laplace type, explicitly defined as double integrals with hypersingular kernels. We prove that such non-local energy forms on a fractal boundary can be approximated by similar non-local energy forms on approximating polygonal curves. Since polygonal curves are one-dimensional, but the fractal boundary typically has Hausdorff dimension greater than one, the approximating energy forms have to be rescaled and modified appropriately to achieve convergence. Our aim is to provide quantitative information on how such rescalings and modifications can be chosen; this is motivated by recent studies of boundary value problems in \cite{CefaloCreoLanciaVernole19, CreoLanciaNazarov20, CreoLanciaNazarovVernole19}. A simple but non-trivial effect, which we call a \enquote{codimension gap}, obstructs a naive solution. We overcome this effect by an averaging technique and the choice of suitable kernels for the approximating non-local forms; the method might be seen as a non-local variant of singular fractal homogenization, \cite{MoscoVivaldi07}. In addition, we consider superpositions of energy forms which are the sum of the Dirichlet integral on the domain plus the non-local boundary energy form. We show that such forms are limits of similar energy forms on approximating polygonal domains. As applications we obtain convergence results for solutions to associated elliptic and parabolic problems with Wentzell boundary conditions.

Non-local energy forms on boundaries appear naturally in connection with trace spaces. Recall that if $B=B(0,1)\subset \mathbb{R}^2$ is the unit disk, then the unique weak solution $u\in H^1(B)$ to the Dirichlet problem for the Laplacian on $B$ with boundary values $\varphi\in H^{1/2}(\partial B)$ on the unit circle $\partial B$ is the Poisson integral $u=\mathbb{H}\varphi$ of $\varphi$. The boundary energy form $\mathcal{Q}^1_{\partial B}(\varphi):=\int_B |\nabla \mathbb{H}\varphi|^2dx$, $\varphi\in H^{1/2}(\partial B)$, on $L^2(\partial B)$ can explicitly be written as Douglas' integral, \cite{Douglas31, Rado30},
\[\mathcal{Q}^1_{\partial B}(\varphi)=\frac{1}{8\pi}\int_0^{2\pi}\int_0^{2\pi}\frac{(\varphi(\theta)-\varphi(\theta'))^2}{\sin^2((\theta-\theta')/2)}d\theta' d\theta \asymp \int_{\partial B}\int_{\partial B}\frac{(\varphi(x)-\varphi(y))^2}{|x-y|^2}dxdy,\quad \varphi\in H^{1/2}(\partial B),\]
see \cite[Section 1.2]{FOT94}. Here $\asymp$ means comparability of quadratic forms. The infinitesimal generator of the boundary energy form $\mathcal{Q}^1_{\partial B}$ is the Dirichlet-to-Neumann operator, \cite{ArendtMazzeo07, ArendtterElst11, CaffarelliSilvestre07}, which in the case of more general domains is a particularly interesting non-local operator of fractional Laplace type related to inverse problems and shape optimization.

In the context of dynamical boundary conditions non-local boundary energy forms describe a jump-type transport along the boundary. The most classical dynamical boundary conditions are local Wentzell conditions involving a diffusion term; they were introduced by Feller, \cite{Feller52, Feller54}, partially motivated by questions in genetics, \cite{Feller51}. Wentzell, \cite{Ventsell59}, then studied all possible boundary conditions that could occur for the generator of a diffusion process in a multidimensional domain. Many authors subsequently investigated dynamical boundary conditions from analytical and probabilistic point of view, see for instance \cite{ArendtKunkelKunze16, BonyCourregePriouret68, Galakhov01, LanciaVernole14, NicaiseLiMazzucato17, Skubachevskii97, Taira14, VogtVoigt03} respectively \cite{GrothausVosshall17, Ikeda61, IkedaWatanabe89, Ishikawa89, SatoUeno65}. In applications Wentzell boundary conditions appear in connection with water waves, \cite{Korman83}, and hydraulic fracturing, \cite{CannonMeyer71}; they are also closely related to effective boundary conditions in ferromagnetics, \cite{HaddarJoly01}. Boundary conditions of Robin-Wentzell type, \cite{Engel03, Favini02, MugnoloRomanelli06, Nittka11, Warma13}, are known to be relevant in acoustics, \cite{GalRuizGoldsteinGoldstein03, Mugnolo06}. Certain non-local boundary conditions have applications in plasma physics, \cite{Bitsadze69, Gurevich12}. Non-local operators of integro-differential type appeared already in early references as summands or special cases, \cite{Ikeda61, Ishikawa89, SatoUeno65}, and received increasing attention in more recent papers \cite{CefaloCreoLanciaVernole19, CreoLanciaNazarovVernole19, GalWarma16a, GalWarma16b, Gurevich12, LanciaVelezVernole16}, where non-local Robin and Wentzell boundary conditions have been considered.

Here we consider purely non-local boundary conditions of, roughly speaking, fractional Laplace type. For simplicity we 
concentrate on this situation and do not add any diffusion term.

We are specifically interested in fractal boundaries. Most of the existing literature on dynamical boundary conditions requires at least Lipschitz regularity, so it seems desirable to provide generalizations to non-Lipschitz domains. Generalized Wentzell conditions appear in transmission problems, \cite{FilocheSapoval00, Sapoval94}, at highly conductive layers, \cite{Lancia02, Mosco94, PhamHuy74}, and numerical experiments support the conjecture that enhanced heat draining may be achieved if the layer is fractal, \cite{CefaloCreoLanciaRodriguez23, LanciaCefaloDellAcqua12}.
Problems with local Wentzell conditions on Koch snowflake boundaries were studied in \cite{HLVT18, LanciaVernole14}, similar problems involving non-local boundary terms in \cite{LanciaVelezVernole16}. Fractal boundaries are idealized limit objects, and only suitable non-fractal shapes approximating them will be computationally tractable. This makes approximations highly relevant, see \cite{HinzMeinert20, HinzMeinert22} for the case of fractal state spaces. In \cite{LanciaVernole14} the convergence of problems with local Wentzell boundary conditions on polygonal pre-fractal domains to limit Wentzell problems on Koch snowflake domains was proved. The papers \cite{CefaloCreoLanciaVernole19, CreoLanciaNazarov20, CreoLanciaNazarovVernole19} contain regularity results and numerical approximations for problems with Wentzell boundary conditions having an additional non-local term on polygonal boundaries. This motivates the question how to provide pre-fractal-to-fractal approximations similar to \cite{LanciaVernole14} for such boundary conditions; there adapted Sobolev spaces on the domain had to be used, \cite[formulas (3.6) and (3.14)]{LanciaVernole14}. For Wentzell problems involving non-local boundary energy forms of Douglas' integral type we can use standard Sobolev spaces on the domain.

Suppose that $\Gamma\subset \mathbb{R}^2$ is the Koch snowflake curve and $\mu$ is the natural arc-wise self-similar Borel probability measure on $\Gamma$. The quadratic forms
\[\mathcal{Q}^1(\varphi)=\int_\Gamma\int_\Gamma\frac{(\varphi(x)-\varphi(y))^2}{|x-y|^d\mu(B(x,|x-y|))}\mu(dx)\mu(dy)
 \asymp \int_\Gamma\int_\Gamma\frac{(\varphi(x)-\varphi(y))^2}{|x-y|^{2d}}\mu(dx)\mu(dy) \]
are comparable to the analog of Douglas' integral for the Koch snowflake domain $\Omega$ enclosed by $\Gamma$. Here $d>1$ is the Hausdorff dimension of $\Gamma$ and the functions $\varphi$ are recruited from the trace space
$B^{2,2}_1(\Gamma)$ of $H^1(\mathbb{R}^2)$ on $\Gamma$, \cite{Jonsson94}. A study of such quadratic forms was provided in \cite{FarkasJacob2001}. If $\Gamma_n$ are the natural pre-fractal polygonal boundaries enclosing the domains $\Omega_n$ and approximating $\Gamma$ as in \cite{LanciaVernole14}, then they all have finite and nonzero one-dimensional Hausdorff measure $\mathcal{H}^1$ and the rescaled measures $\mu_n:=\frac{3^n}{4^n}\mathcal{H}^1(\cdot\cap \Gamma_n)$ converge weakly to $\mu$. For smooth functions $u$ on $\mathbb{R}^2$ we can observe that 
\[ \int_{\Gamma_n}\int_{\Gamma_n}\frac{(u(\xi)-u(\eta))^2}{|\xi-\eta|^{d'}\mu_n(B(\xi,|\xi-\eta|))}\mu_n(d\xi)\mu_n(d\eta)\to \int_\Gamma\int_\Gamma\frac{(u(x)-u(y))^2}{|x-y|^{d'}\mu(B(x,|x-y|))}\mu(dx)\mu(dy)\]
as $n\to \infty$. If the approximating boundary energy forms are supposed to have the natural trace spaces $H^{1/2}(\Gamma_n)$ of $H^1(\Omega_n)$ as domains, then $d'=1$ must be chosen and the desirable $d'=d>1$ is out of bounds. The parameter $d'$ should be viewed as a
smoothness parameter; this can be seen from well-known results for non-local Dirichlet forms on metric measure spaces, \cite{ChenKumagai08}. The smaller codimension $2-d$ of the fractal curve $\Gamma$ in comparison to $2-1$ in the case of the polygons $\Gamma_n$ results in a higher smoothness parameter for the trace spaces, \cite{Jonsson94, JonssonWallin84}, and this is what we mean by \enquote{codimension gap}. Here our workaround is to approximate $\mathcal{Q}^1$ by boundary energy forms 
\begin{multline}
\mathcal{Q}^1_n(\varphi)=\frac{3^n}{4^n}\int_{\Gamma_n}\int_{\Gamma_n\cap B(\xi,C\: 3^{-n})^c}\frac{(\varphi(\xi)-\varphi(\eta))^2}{|\xi-\eta|^{1+d}}\mathcal{H}^1(d\xi)\mathcal{H}^1(d\eta)\notag\\
+\frac{3^n}{4^n}\int_{\Gamma_n}\int_{\Gamma_n\cap B(\xi,C\: 3^{-n})}\frac{(\varphi(\xi)-\varphi(\eta))^2}{|\xi-\eta|^{2}}\mathcal{H}^1(d\xi)\mathcal{H}^1(d\eta);
\end{multline}
where $C>0$ is a fixed constant. For large interaction ranges the desired exponent $d'=d$ appears in the kernel, for short range interactions we keep $d'=1$. In the limit for $n\to \infty$ the first summand takes over, the second vanishes, and we recover $\mathcal{Q}^1$.

Instead of $\mathcal{Q}^1$ and $\mathcal{Q}^1_n$, defined in terms of double integrals with explicit kernels, one could of course study the exact analogs of Douglas' integral, abstractly defined using Dirichlet form theory, \cite{BH91, FOT94}. Then the convergence can actually be proved much easier. Such an abstract point of view is very suitable for compactness and stability results, \cite{HR-PT21, HR-PT23}, but in general it cannot be expected to give much explicit information. Numerical investigations, on the other hand, require a high level of detail with truly explicit formulas, as for instance achieved in \cite{Achdou06, LanciaCefaloDellAcqua12, NicaiseLiMazzucato17} for other mathematical questions. Our results here are a middle ground, maybe somewhat similar to \cite{Mosco13}: We study boundary energy forms comparable to the exact analogs of Douglas' integral, but explicitly defined as double integrals. This reveals appropriate rescalings and the asymptotics of constants. In the special case of the Koch snowflake our results are fully explicit. 

In Theorem \ref{T:Moscoboundary} we prove a result that contains the Mosco convergence, \cite{Mosco94}, of the forms $\mathcal{Q}^1_n$ to $\mathcal{Q}^1$ as a very special case. Since the boundaries and measures give rise to varying Hilbert spaces, we formulate it in the framework of \cite{KuwaeShioya03}, see \cite[Definition 2.11]{KuwaeShioya03} or Appendix \ref{S:Notions}. Related results can be found in \cite{HinzTeplyaev15, MoscoVivaldi07, PostSimmer21}. Our proof of Theorem \ref{T:Moscoboundary} is both a non-trivial extension and a non-uniform refinement of an averaging method used in \cite{Hinz09}. There the method was applied to $d$-sets. Here the curve $\Gamma$ is an arbitrary quasicircle, which may have parts of different Hausdorff dimensions. In \cite[Theorem 1.2]{Rohde01} it was shown how to construct measures with refined doubling conditions for any quasicircle, see \cite{Wu98} for related arguments. Because these refined doubling conditions are just what a suitable trace theorem requires, \cite[Theorem 1]{Jonsson94}, we can formulate Theorem \ref{T:Moscoboundary} for non-local energy forms on more general quasicircles. In particular, $\Gamma$ does not have to be a Koch snowflake curve. The proof of Theorem \ref{T:Moscoboundary} does not make a direct use of quasiconformal parametrizations, but of metric consequences, \cite{Ahlfors63, Lehto87, MartioSarvas79, Vaisala88}. The approximation of energy forms on closed quasidiscs, \cite{Gehring82},  by similar energy forms on closed polygonal domains is verified in Theorem \ref{T:MoscoWentzell}.  Intermediate results are Theorem \ref{T:adhoc}, where we provide a possible construction of approximating polygonal $(\varepsilon,\infty)$-uniform domains with common parameter $\varepsilon>0$, and Theorem \ref{T:Mosco}, where we show the Mosco convergence of Dirichlet integrals on these domains. The convergence of solutions to elliptic and parabolic problems with non-local boundary conditions on approximating polygonal domains to corresponding solutions on limit quasidiscs is observed in Theorems \ref{T:ellipticstability} and \ref{T:parabolicstability} respectively.

We proceed as follows: In Section \ref{S:Quasi} we state and discuss our standing assumptions. In Section \ref{S:Polygons} we investigate sequences of approximating polygonal curves and measures defined by averaging. We introduce non-local energy forms on quasicircles and polygonal curves in Section \ref{S:Non-local} and prove their Mosco-convergence in Section \ref{S:Moscoboundary}. Section \ref{S:Disks} contains an approximation scheme for quasidiscs in terms of polygonal $(\varepsilon,\infty)$-uniform domains. In Sections \ref{S:Mosco}, \ref{S:MoscoWentzell} and \ref{S:Apps} we provide results on the Mosco-convergence of Dirichlet integrals and superpositions and show brief applications to elliptic and parabolic equations.

By $\mathcal{L}^2$ and  $\mathcal{H}^s$ we denote the $2$-dimensional Lebesgue measure and the $s$-dimensional Hausdorff measure on $\mathbb{R}^2$, respectively. Given a Borel measure $\mu$ on $\mathbb{R}^2$, a Borel set $E\subset \mathbb{R}^2$ with $\mu(E)>0$ and a Borel function $f$ that is $\mu$-integrable over $E$, we use the notation $\fint_E f\:d\mu:=\frac{1}{\mu(E)}\int_E f\:d\mu$.

\section{Quasicircles and measures}\label{S:Quasi}

Recall that a \emph{quasicircle} $\Gamma\subset \mathbb{R}^2$ is the image of a circle under a quasiconformal map of the plane $\mathbb{R}^2$ onto itself, \cite{Ahlfors63, Gehring82, Jones81, Lehto87}.
Our standing assumptions in this article are that $\Gamma\subset \mathbb{R}^2$ is a quasicircle, that
\begin{equation}\label{E:sd}
1\leq d\leq s<2
\end{equation}
and $c_{\mu}>1$ are constants and that $\mu$ is a Borel probability measure on $\mathbb{R}^2$ with $\supp\mu=\Gamma$ satisfying
\begin{equation}\label{E:refineddoubling}
c_{\mu}^{-1}k^d\mu(B(x,r))\leq \mu(B(x,kr)\leq c_{\mu} k^s\mu(B(x,r)),\quad x\in \Gamma,\ r>0,\ k>1,\ kr\leq \diam\Gamma,
\end{equation}
and
\begin{equation}\label{E:lowerboundmu}
\mu(B(x,1))\geq c_{\mu}^{-1},\quad x\in \Gamma.
\end{equation}

\begin{examples}\mbox{}\label{Ex:measure}
\begin{enumerate}
\item[(i)] If $\Gamma$ is rectifiable, then (\ref{E:refineddoubling}) and (\ref{E:lowerboundmu}) hold for $\mu=\mathcal{H}^1(\Gamma\cap \cdot)/\mathcal{H}^1(\Gamma)$ with $s=d=1$ in (\ref{E:sd}).
\item[(ii)] If $\Gamma$ is the classical arc-wise self-similar Koch snowflake curve with contraction ratio $1/3$, \cite{Falconer90, Mattila95}, then the arc-wise self-similar probability measure $\mu$ satisfies (\ref{E:refineddoubling}) with $d=s=\log 4/\log 3$. For variants with contraction ratio $1/4\leq p< 1/2$ the same is true with $d=s=\log 4/(-\log p)$. If $\Gamma$ is a homogeneous scale irregular snowflake-like curve \cite{Capitanelli10, Mosco02, Rohde01} and a law of large numbers holds for the sequence of scaling factors, then there is a Borel probability measure $\mu$ satisfying (\ref{E:refineddoubling}) suitable $1\leq d=s< 2$, \cite[p. 1224]{Capitanelli10}. Also (\ref{E:lowerboundmu}) holds in all these cases.
\item[(iii)] In \cite[Theorem 1.2 and its proof]{Rohde01} it was shown that for any quasicircle $\Gamma$ one can construct a Borel probability measure $\mu$ on $\mathbb{R}^2$ with $\supp\mu=\Gamma$ such that (\ref{E:refineddoubling}) and 
(\ref{E:lowerboundmu}) hold with $d=1$ and some $1\leq s<2$ in (\ref{E:sd}). As pointed out,  \cite[p. 645]{Rohde01}, the  measures $\mu$ constructed there are generally not canonical in any way. A particularly interesting class of quasicircles are the general snowflake-like curves in \cite{Rohde01}; any quasicircle is the image of such a curve under a bi-Lipschitz map of the plane onto itself, \cite[Theorem 1.1]{Rohde01}.
\end{enumerate}
\end{examples}

\begin{remark}\label{R:doubling}\mbox{}
\begin{enumerate}
\item[(i)] Property (\ref{E:refineddoubling}) implies that $\mu$ is volume doubling in the usual sense, more precisely,
\begin{equation}\label{E:mudoubling}
\mu(B(x,2r))\leq c_{\mu,D}\mu(B(x,r)),\quad x\in \Gamma,\quad 0<r\leq (\diam\Gamma)/2,
\end{equation}
with a constant $c_{\mu,D}>1$.
\item[(ii)] Property (\ref{E:refineddoubling}) also implies that  
\begin{equation}\label{E:uplowreg}
c^{-1}r^s\leq \mu(B(x,r))\leq c\:r^d,\quad x\in \Gamma,\quad 0<r\leq \diam\Gamma,
\end{equation}
with a suitably readjusted constant $c>1$. In particular, the Hausdorff dimension of $\Gamma$ is bounded by $s$ and bounded below by $d$. 
\item[(iii)] If $\mu$ satisfies (\ref{E:uplowreg}) with $d=s$, then it is said to be \emph{$d$-regular}; in this case 
it also satisfies (\ref{E:refineddoubling}) with $d=s$. The measures in Examples \ref{Ex:measure} (i) and (ii) are $d$-regular.
\item[(iv)] The scaling properties in (\ref{E:refineddoubling}) allow the local dimension of $\Gamma$ to fluctuate 
continuously. They have been investigated by various authors, see for instance \cite{Assouad80,BylundGudayol00,Dynkin84,LukkainenSaksman98,VolbergKonyagin87}. In \cite{Jonsson94} they have been used as geometric hypotheses for trace and extension results for Besov spaces.
\end{enumerate}  
\end{remark}

\section{Polygonal approximations and measures}\label{S:Polygons}

We agree to take all subarcs $a$ of $\Gamma$ to be closed in the relative topology of $\Gamma$. We call two subarcs of $\Gamma$ \emph{disjoint} if their interiors are disjoint. To a collection of disjoint subarcs $a_n$ of $\Gamma$ that cover $\Gamma$ we refer as a \emph{partition} of $\Gamma$.

\begin{assumption}\mbox{}\label{A:basicass}
We assume that $0<p,q<1$, $M>1$, $\omega=(\omega_1,\omega_2,...)\in \{p,q\}^\mathbb{N}$ and $(\mathcal{I}_n)_{n\geq 0}$ is a sequence of finite partitions $\mathcal{I}_n=\{a_{n,j}\}_j$ of $\Gamma$ such that for all $n$ and $j$ we have 
\begin{equation}\label{E:basicass}
\omega_1\cdots \omega_n\diam\Gamma\leq \diam a_{n,j}<M\:\omega_1\cdots \omega_n \diam\Gamma.
\end{equation}
\end{assumption}

\begin{remark}
Assumption \ref{A:basicass} ensures that arcs in a single partition have about the same diameter. This metric uniformity makes the approximation method in Section \ref{S:Moscoboundary} tractable, and it seems more important for this method than a uniform local dimension.
\end{remark}

\begin{examples}\label{Ex:basicass}\mbox{}
\begin{enumerate}
\item[(i)] For homogeneous snowflake-like curves, \cite{Rohde01}, one can choose natural partitions. Given $1/4\leq p< 1/2$, consider the polygonal curve in $\mathbb{R}^2$ having four edges of length $p$ and vertices $(0,0)$, $(p,0)$, $(1/2,h_p)$, $(1-p,0)$ and $(1,0)$, where $h_p^2=p^2-(1/2-p)^2$. To any scaled copy of this polygonal curve we refer as a \emph{Koch type segment with parameter $p$}. 

Suppose that $\Gamma$ is the classical arc-wise self-similar Koch snowflake, \cite{Falconer90, LanciaVernole14}, and that it is constructed starting from an equilateral triangle with edge length one by successive and simultaneous replacement of edges by 'outward pointing' Koch type segments with parameter $3^{-n}$. At stage $n$, let $V_n$ be the vertex set of the resulting polygon with edges of length $3^{-n}$ and let $\mathcal{I}_n$ be the natural partition of $\Gamma$ into the arcs connecting two neighboring points from $V_n$. Then Assumption \ref{A:basicass} holds with $p=q=1/3$ and these very $\mathcal{I}_n$. Assumption \ref{A:basicass} is also satisfied for the obvious modification using a more general parameter $p$ as indicated.
\begin{figure}[h]
\begin{minipage}{0.45\textwidth}
\begin{center}
\scalebox{.3}{
  	\def\HEX{ \put(-18,10){{\color{black} \circle*{1.0}}}
  	\put(18,10){{\color{black} \circle*{1.0}}}
  	\put(0,-20.0){{\color{black} \circle*{1.0}}}
  		}
  	\begin{picture}(324, 360)(-162, -180)
  	\put(-175,-190)
  	{\includegraphics[width=350pt,height=380pt]{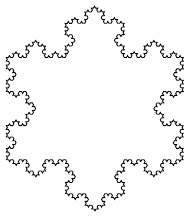}}
  	\setlength{\unitlength}{9pt}
  	\HEX
  	\end{picture}
  	}
  	\caption{Points of $V_0$ enclosing arcs of diameter $1$.}
  	\label{F:V0}  	
  	\end{center}
  	\end{minipage}
\begin{minipage}{0.45\textwidth}
\begin{center}
\scalebox{.3}{
  	\def\HEX{ 
  		\put(-12,0){{\color{black} \circle*{1.0}}}
  		\put(12,0){{\color{black} \circle*{1.0}}}
  		\put(-6,10){{\color{black} \circle*{1.0}}}
  		\put(-18,10){{\color{black} \circle*{1.0}}}
  		\put(0,20.0){{\color{black} \circle*{1.0}}}
  		\put(0,-20.0){{\color{black} \circle*{1.0}}}
  		\put(6,10){{\color{black} \circle*{1.0}}}
  		\put(18,10){{\color{black} \circle*{1.0}}}
  		\put(-6,-10){{\color{black} \circle*{1.0}}}
  		\put(-18,-10){{\color{black} \circle*{1.0}}}
  		\put(6,-10){{\color{black} \circle*{1.0}}}
  		\put(18,-10){{\color{black} \circle*{1.0}}}
  		}
  	\begin{picture}(324, 360)(-162, -180)
  	\put(-175,-190)
  	{\includegraphics[width=350pt,height=380pt]{snowflake-c.pdf}}
  	\setlength{\unitlength}{9pt}
  	\HEX
  	\end{picture}
  	}
  	\caption{Points of $V_1$ enclosing arcs of diameter $3^{-1}$.}
  	\label{F:V1}  	
  	\end{center}
  	\end{minipage}
  \end{figure} 
  
If $1/4\leq q< p<1/2$ and at each stage $n$ all edges are replaced by a Koch type segment with the same parameter $\omega_n\in \{p,q\}$, then we recover the homogeneous scale irregular snowflake-like curves in \cite{Capitanelli10, Mosco02} and \cite{Rohde01}, and with an analogous natural partition Assumption \ref{A:basicass} is satisfied.
\item[(ii)] Recall that a \emph{Jordan curve} $\Gamma\subset \mathbb{R}^2$ is the image of a circle under an injective continuous map into $\mathbb{R}^2$. For any Jordan curve $\Gamma\subset \mathbb{R}^2$ and sufficiently small $0<q<1$ one can find a sequence of finite partitions $(\mathcal{I}_n)_{n\geq 0}$ of $\Gamma$ into arcs having the same diameter, \cite{AltBeer35, Schoenberg40},
see  \cite[Section 2.5]{HerronMeyer12} for detailed comments. This suffices to satisfy Assumption \ref{A:basicass} with $p=q$ and $M=2$.
\end{enumerate}
\end{examples}

Let Assumption \ref{A:basicass} be satisfied and let the $\mathcal{I}_n=\{a_{n,j}\}_j$ be as there. 
For fixed $n$ and $j$, let $e_{n,j}\subset \mathbb{R}^2$ denote the line segment connecting the two endpoints of the arc $a_{n,j}$ and let $\Gamma_n:=\bigcup_j e_{n,j}$ denote the closed polygonal curve obtained as the union of these line segments. Since each edge $e_{n,j}$ of $\Gamma_n$ is contained in the convex hull of $a_{n,j}$, we have
\begin{equation}\label{E:convex}
\max_{x\in a_{n,j}, w\in e_{n,j}}|x-w|\leq \diam a_{n,j}.
\end{equation}

Recall that since $\Gamma$ is a quasicircle, it satisfies the \emph{bounded turning condition} (also called \emph{arc condition})
\begin{equation}\label{E:3point}
S:=\sup_{x,y\in \Gamma,\ x\neq y}\frac{\diam a(x,y)}{|x-y|}<+\infty,
\end{equation}
where $a(x,y)$ denotes a subarc of $\Gamma$ of minimal diameter connecting $x$ and $y$. It is well known that this condition characterizes quasicircles within the class of planar Jordan curves, \cite{Ahlfors63}, see for instance \cite[Chapter I, Section 6.5]{Lehto87} for details. 

Given two closed subsets $F_1$ and $F_2$ of $\mathbb{R}^2$, their Hausdorff distance is defined as
\[d_H(F_1,F_2):=\inf\left\lbrace \varepsilon>0: F_1\subset (F_2)_\varepsilon\ \text{and}\ F_2\subset (F_1)_\varepsilon\right\rbrace,\]
where for each closed $F\subset \mathbb{R}^2$ and each $\varepsilon>0$ we write $(F)_\varepsilon:=\{x\in\mathbb{R}^2: \dist(x,F)\leq \varepsilon\}$ for the closed $\varepsilon$-parallel set of $F$.

\begin{lemma}\mbox{}\label{L:Hausdorffconvbd} Let Assumption \ref{A:basicass} be in force.
\begin{enumerate}
\item[(i)] The polygonal curves $\Gamma_n$ converge to $\Gamma$ in the Hausdorff sense, more precisely
\[d_H(\Gamma_n,\Gamma)<M\:\omega_1\cdots \omega_n \diam\Gamma.\]
\item[(ii)] For each $n$ and each edge $e_{n,j}$ of $\Gamma_n$ we have 
\begin{equation}\label{E:edgelength}
S^{-1}\diam a_{n,j}\leq \mathcal{H}^1(e_{n,j})\leq \diam a_{n,j}
\end{equation}
\end{enumerate}
\end{lemma}

\begin{proof}
Item (i) follows from (\ref{E:basicass}) and (\ref{E:convex}), item (ii) from (\ref{E:convex}) and (\ref{E:3point}).
\end{proof}

\begin{examples}\label{Ex:Kochagain}
If $\Gamma$ is the classical Koch snowflake curve, endowed with the natural partitions $\mathcal{I}_n$ as mentioned in Example \ref{Ex:basicass} (i), then the polygonal curve $\Gamma_n$ is the familiar pre-fractal at stage $n$.

\begin{figure}[h]
\begin{minipage}{0.45\textwidth}
\begin{center}
\scalebox{.3}{
  	\def\HEX{ \put(-18,10){\line(1, 0){36}}
  	\put(-18,10){\line(3, -5){18}}
  	\put(18,10){\line(-3, -5){18}}
  		}
  	\begin{picture}(324, 360)(-162, -180)
  	\put(-175,-190)
  	{\includegraphics[width=350pt,height=380pt]{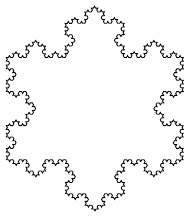}}
  	\setlength{\unitlength}{9pt}
  	\HEX
  	\end{picture}
  	}
  	\caption{Koch snowflake $\Gamma$ and polygon $\Gamma_0$.}
  	\label{F:Gamma0}  	
  	\end{center}
  	\end{minipage}
\begin{minipage}{0.45\textwidth}
\begin{center}
\scalebox{.3}{
  	\def\HEX{ \put(-18,10){\line(1, 0){12}}
  	\put(18,10){\line(-1, 0){12}}
   	\put(-6,10){\line(3, 5){6}}
  	\put(6,10){\line(-3, 5){6}}
  	\put(-18,10){\line(3, -5){6}} 
  	\put(-12,0){\line(-3, -5){6}}
  	\put(12,0){\line(3, 5){6}}
  	\put(12,0){\line(3, -5){6}}
  	\put(-6,-10){\line(3, -5){6}}
  	\put(6,-10){\line(-3, -5){6}}
  	\put(-18,-10){\line(1, 0){12}}
  	\put(18,-10){\line(-1, 0){12}}
  		}
  	\begin{picture}(324, 360)(-162, -180)
  	\put(-175,-190)
  	{\includegraphics[width=350pt,height=380pt]{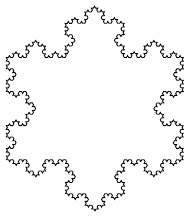}}
  	\setlength{\unitlength}{9pt}
  	\HEX
  	\end{picture}
  	}
  	\caption{Koch snowflake $\Gamma$ and polygon $\Gamma_1$.}
  	\label{F:Gamma1}  	
  	\end{center}
  	\end{minipage}
  \end{figure} 

\end{examples}

Since $\Gamma$ is a quasicircle, there is some $0<\theta\leq 1$ such that for all $\xi\in \Gamma$ and $r>0$ 
\begin{equation}\label{E:theta}
\text{only the connected component of $\Gamma\cap \overline{B(\xi,r)}$ that contains $\xi$ intersects $B(\xi,\theta r)$.}
\end{equation} 
This is another of the many conditions that characterize quasicircles within the class of planar Jordan curves, see for instance \cite[2.28. Remark]{MartioSarvas79}.

We call two subarcs $a$ and $a'$ in a partition $\mathcal{I}_n$ \emph{adjacent} if they share an endpoint. By (\ref{E:refineddoubling}), (\ref{E:basicass}) and (\ref{E:theta}) two adjacent subarcs have comparable measure.
\begin{lemma}\label{L:unicompare}
Let Assumption \ref{A:basicass} be in force. Then for any $n\geq 1$ and any adjacent $a,a'\in \mathcal{I}_n$ we have 
\begin{equation}\label{E:unicompare}
\mu(a')\leq c_\mu\Big(\frac{8MS}{\theta}\Big)^s\:\mu(a).
\end{equation}
\end{lemma}
\begin{proof}
Suppose that $a\in \mathcal{I}_n$ and $x,y\in V_n$ are such that $a=a(x,y)$. Let $z\in a$ be such that $2r:=|z-x|=|z-y|$. Then the connected component of $\Gamma\cap \overline{B(z,r)}$ containing $z$ is $a\cap \overline{B(z,r)}$, and by (\ref{E:theta}) it follows that $\Gamma\cap B(z,\theta r)=a\cap B(z,\theta r)$. Consequently 
\[\mu(B(z,\theta r))=\mu(a\cap B(z,\theta r))\leq \mu(a).\]
Now suppose that $a'\in\mathcal{I}_n$ is adjacent to $a$. Then
\[a\cup a'\subset B(z,2M\omega_1\cdots \omega_n\diam\Gamma)\quad \text{and}\quad \frac{1}{4S} \omega_1\cdots \omega_n\diam\Gamma\leq r\]
by (\ref{E:basicass}), and by (\ref{E:refineddoubling}) therefore
\[\mu(a')\leq \mu(a\cup a')\leq \mu(B(z,8MSr))\leq c_\mu\Big(\frac{8MS}{\theta}\Big)^s\:\mu(B(z,\theta r))\leq c_\mu\Big(\frac{8MS}{\theta}\Big)^s\:\mu(a).\]
\end{proof}

Generalizing the notation used in Examples \ref{Ex:basicass} (i), let $V_n\subset \Gamma$ denote the finite set of all endpoints of arcs $a_{n,j}$ in $\mathcal{I}_n$; clearly $V_n\subset \Gamma_n$. We write $V_\ast:=\bigcup_{n\geq 0} V_n$. 

\begin{remark}
In the case of the homogeneous snowflake-like curves with natural partitions as in Examples \ref{Ex:Kochagain} (i) the sets of endpoints increase monotonically, $V_n\subset V_{n+1}$, $n\geq 0$. Also the partitions used in \cite[Theorem 1.2]{Rohde01} to construct measures $\mu$ satisfying (\ref{E:refineddoubling}) have this property. Here we do not insist on it.
\end{remark}

Due to (\ref{E:sd}) and (\ref{E:refineddoubling}) the measure $\mu$ is atom free and in particular, $\mu(V_\ast)=0$. We introduce Borel probability measures $\mu_n$ on the approximating polygonal curves $\Gamma_n$ by 
\begin{equation}\label{E:disint}
\mu_n(B):=\sum_{j}\frac{\mu(a_{n,j})}{\mathcal{H}^1(e_{n,j})}\mathcal{H}^1(B\cap e_{n,j}),\quad \text{$B\subset \Gamma_n$ Borel}.\end{equation}

\begin{remark}
The measures $\mu_n$ are comparable to $\mathcal{H}^1(\cdot \cap\Gamma_n)$: Since $0<\min_j\mu(a_{n,j})\leq \mu(a_{n,j})\leq 1$ for all $j$, (\ref{E:basicass}) and (\ref{E:edgelength}) imply that
\begin{equation}\label{E:absolute}
\frac{\min_j \mu(a_{n,j})}{M\omega_1\cdots \omega_n\diam\Gamma}\mathcal{H}^1(\cdot \cap \Gamma_n)\leq \mu_n\leq \frac{\#(\mathcal{I}_n) S}{\omega_1\cdots \omega_n\diam\Gamma}\mathcal{H}^1(\cdot \cap \Gamma_n),
\end{equation}
where $\#(\mathcal{I}_n)$ denotes the cardinality of $\mathcal{I}_n=\{a_{n,j}\}_j$. In particular, the density $d\mu_n/d\mathcal{H}^1(\cdot \cap\Gamma_n)$ is bounded and bounded away from zero. The comparison in (\ref{E:absolute}) implies that for any $0\leq  p\leq +\infty$ we have 
$L^p(\Gamma_n,\mu_n)=L^p(\Gamma_n,\mathcal{H}^1(\cdot\cap\Gamma_n))$.
\end{remark}

\begin{examples}\label{Ex:Kochagain2}
If $\Gamma$ is the classical Koch snowflake curve, endowed with the natural partitions $\mathcal{I}_n$ and $\mu$ is the natural arc-wise self-similar probability measure, then $\mu(a_{n,j})=4^{-n}$, $\mathcal{H}^1(e_{n,j})=3^{-n}$, and consequently
\begin{equation}\label{E:munKochcase}
\mu_n=\frac{3^n}{4^n}\:\mathcal{H}^1(\cdot \cap \Gamma_n).
\end{equation}
\end{examples}

The edge-wise averaging inherent to the measures $\mu_n$ can be shifted to functions. For each point $x\in \Gamma\setminus V_\ast$ and each $n$ let $a_n(x)$ be the unique arc in $\mathcal{I}_n$ such that $x\in a_n(x)\setminus V_n$ and $e_n(x)$ the edge of $\Gamma_n$ having the same endpoints as $a_n(x)$. We can rewrite (\ref{E:disint}) as
\begin{equation}\label{E:defmun}
\mu_n(B)=\int_\Gamma \frac{1}{\mathcal{H}^1(e_n(x))}\:\mathcal{H}^1(B\cap e_n(x))\:\mu(dx),\quad \text{$B\subset \Gamma_n$ Borel}.
\end{equation}
Given $\varphi\in L^1(\Gamma_n,\mu_n)$, we define a function $[\varphi]_n$ on $\Gamma\setminus V_\ast$ by edge-wise averaging on $\Gamma_n$, 
\begin{equation}\label{E:deffn}
[\varphi]_n(x):=\fint_{e_n(x)}\varphi(w)\:\mathcal{H}^1(dw),\quad x\in \Gamma\setminus V_\ast.
\end{equation}
Identity (\ref{E:defmun}) and standard approximation then imply that
\begin{equation}\label{E:trading}
\int_{\Gamma_n} \varphi(w)\:\mu_n(dw)=\int_\Gamma[\varphi]_n(x) \mu(dx).
\end{equation}

We have $\lim_{n\to \infty}\mu_n= \mu$ in the weak sense: If $u\in C_b(\mathbb{R}^2)$, then $u(x)=\lim_{n\to \infty} [u]_n(x)$
at any $x\in \Gamma\setminus V_\ast$, and (\ref{E:trading}) and bounded convergence give
\begin{equation}\label{E:weakconveval}
\lim_{n\to \infty} \int_{\Gamma_n} u\:d\mu_n=\int_\Gamma u\:d\mu.
\end{equation}
Let $E_\Gamma:\lip(\Gamma)\to \lip_b(\mathbb{R}^2)$ denote the Whitney extension operator acting on the space $\lip(\Gamma)$ of Lipschitz functions on $\Gamma$ and taking values in the space $\lip_b(\mathbb{R}^2)$ of bounded Lipschitz functions on $\mathbb{R}^2$, \cite[Chapter VI, Theorem 3]{Stein70}. Using (\ref{E:weakconveval}) we find the following result on the convergence of Hilbert spaces in the sense of \cite[Section 2.2]{KuwaeShioya03}, see Appendix \ref{S:Notions}.
\begin{lemma}\label{L:KSconvbd}  Let Assumption \ref{A:basicass} be satisfied. Then for any $\varphi\in \lip(\Gamma)$ we have
\[\lim_{n\to \infty} \left\| E_\Gamma \varphi\right\|_{L^2(\Gamma_n,\mu_n)}=\left\|\varphi\right\|_{L^2(\Gamma,\mu)}.\]
With identification maps $\varphi\mapsto (E_\Gamma \varphi)|_{\Gamma_n}$, $\varphi\in \lip(\Gamma)$, the sequence of Hilbert spaces $(L^2(\Gamma_n,\mu_n))_n$ converges to $L^2(\Gamma,\mu)$.
\end{lemma}

For later use we prove a lemma that refines (\ref{E:absolute}) in the sense that although the original measure $\mu$ may have local dimensional fluctuations, the  measures $\mu_n$ are locally $1$-regular and volume doubling in a \emph{uniformly scale dependent sense}. To formulate it, recall from \cite[Lemma 4.1]{Rohde01} that there are constants $c>0$ and $\gamma>0$ such that for any $R>0$ any arc $a\subset \Gamma$ contains at most a number
\begin{equation}\label{E:NR}
N(R)\leq c\:R^{\gamma}
\end{equation}
of disjoint subarcs $a_1',...,a_{N(R)}'$ of diameter at least $(\diam a)/R$. Given $n\geq 1$ and $\xi\in \Gamma_n\setminus V_n$, let $e_n(\xi)$ be the unique edge of $\Gamma_n$ containing $\xi$ and let $a_n(\xi)$ be the arc in $\mathcal{I}_n$ having the same endpoints as $e_n(\xi)$. 

We use the shortcut notation
\begin{equation}\label{E:rn}
r_n:=\frac{1}{2MS}\omega_1\cdots \omega_n\diam \Gamma,\quad n\geq 1.
\end{equation}
Note that $\lim_{n\to \infty} r_n=0$ by Assumption \ref{A:basicass}.

\begin{lemma}\label{L:mundoubling}  Let Assumption \ref{A:basicass} be satisfied.
Given $C>1$, let $N=N(C+2MS)$ be as in (\ref{E:NR}). There is some $n_0\geq 1$ such that
\begin{equation}\label{E:mundreg}
\frac{\mu(a_n(\xi))}{2CM^2S r_n}\:r\leq \mu_n(B(\xi,r))\leq \frac{N c_\mu^N \Big(\frac{8MS}{\theta}\Big)^{Ns} \mu(a_n(\xi))}{Mr_n}\:r
\end{equation}
for all $n\geq n_0$, $\xi\in \Gamma_n\setminus V_n$ and $0<r<CM\:r_n$. In particular,
\begin{equation}\label{E:mundoubling}
\mu_n(B(\xi,2r))\leq c_{D}\mu_n(B(\xi,r))
\end{equation}
for all $n\geq n_0$, $\xi\in \Gamma_n\setminus V_n$ and $0<r<CM\:r_n/2$ with a constant  $c_{D}>1$ that does not depend on $n$, $\xi$ or $r$.
\end{lemma}

\begin{examples}\label{Ex:Kochagain3}
If $\Gamma$ is the classical Koch snowflake, endowed with the natural partitions $\mathcal{I}_n$ the natural arc-wise self-similar probability measure $\mu$, then $\omega_1\cdots \omega_n=3^{-n}$ and $\mu(a_n(\xi))=4^{-n}$ for all $\xi\in \Gamma_n\setminus V_n$, and (\ref{E:mundreg}) reduces to 
\[c^{-1}\frac{3^n}{4^n} r\leq \mu_n(B(\xi,r))\leq c\frac {3^n}{4^n} r,\quad 0<r<3^{-n},\] 
as expected from (\ref{E:munKochcase}).
\end{examples}

\begin{proof}[Proof of Lemma \ref{L:mundoubling}.]
It suffices to verify the two inequalities in (\ref{E:mundreg}), condition (\ref{E:mundoubling}) then follows. If $0<r<\frac{1}{2S}\omega_1\cdots\omega_n\diam\Gamma$, then $r<\mathcal{H}^1(e_n(\xi))/2$ by (\ref{E:basicass}) and (\ref{E:edgelength}). This gives $\mathcal{H}^1(B(\xi,r)\cap e_n(\xi))\geq r$. If now $0< r<CMr_n$, then we can use (\ref{E:basicass}), (\ref{E:edgelength}) and (\ref{E:disint}) to see that
\[\mu_n(B(\xi,r))\geq \mu_n(B(\xi,r/C))\geq \frac{\mu(a_n(\xi))}{M\omega_1\cdots \omega_n\diam\Gamma}\mathcal{H}^1(B(\xi,r/C)\cap e_n(\xi))\geq \frac{\mu(a_n(\xi)) r}{CM\omega_1\cdots \omega_n\diam\Gamma},\]
which is the lower bound in (\ref{E:mundreg}). To verify also the upper bound in (\ref{E:mundreg}), let $n_0$ be large enough to have $\max\{Cr_n,d_H(\Gamma,\Gamma_n)\}\leq \diam\Gamma/100$ for all $n\geq n_0$. Let such $n$ be fixed and $0< r<CMr_n$. Then there must be some $z\in \Gamma_n\setminus B(\xi,r)^c$. Starting at $z$, follow $\Gamma_n$ clockwise. Let $x$ be the last point of $V_n$ visited by $\Gamma_n$ after starting at $z$ and before entering $B(\xi,r)$ for the first time. Let $y$ be the first point of $V_n$ visited by $\Gamma_n$ after leaving $B(\xi,r)$ for the last time before returning to $z$. By (\ref{E:3point}) 
the arc $a(x,y)\subset\Gamma$ from $x$ to $y$ has diameter
\[\diam a(x,y)\leq S|x-y|\leq S(2r+2\max_j\diam e_{n,j})\leq (C+2MS)\omega_1\cdots\omega_n\diam\Gamma.\]
By (\ref{E:basicass}) and the remarks preceding the lemma, the arc $a(x,y)$ can contain at most $N=N(C+2MS)$ subarcs $a'_1,...,a'_{N}$ from $\mathcal{I}_n$. Since $x,y\in V_n$, the corresponding line segments $e_1',...,e_N'\subset \Gamma_n$ are the only segments of $\Gamma_n$ that intersect $B(\xi,r)$. From (\ref{E:basicass}), (\ref{E:edgelength}) and (\ref{E:unicompare})
it now follows that 
\[\mu_n(B(\xi,r))\leq \sum_{i=1}^N\frac{\mu(a_i')}{\mathcal{H}^1(e_i')}\mathcal{H}^1(e_i'\cap B(\xi,r))\leq \frac{N c_\mu^N \Big(\frac{8MS}{\theta}\Big)^{Ns} \mu(a_n(\xi))\:2r}{S^{-1}\omega_1\cdots\omega_n\diam\Gamma}.\]
\end{proof}

\section{Non-local energy forms on curves}\label{S:Non-local}

Given 
\begin{equation}\label{E:constellation}
\frac{2-d}{2}<\alpha<1+\frac{2-s}{2}
\end{equation}
and a Borel function $\sigma_\alpha:\mathbb{R}^2\times (0,+\infty)\to (0,+\infty)$ such that 
\begin{equation}\label{E:symbol}
c^{-1}r^{2\alpha-2}\mu(B(x,r))\leq \sigma_\alpha(x,r)\leq c\:r^{2\alpha-2}\mu(B(x,r)),\quad x\in \Gamma,\quad 0<r\leq \diam\Gamma
\end{equation}
with $c>1$ independent of $x$ and $r$, we can define a quadratic form $\mathcal{Q}^\alpha$ on $L^2(\Gamma,\mu)$ by 
\begin{equation}\label{E:boundaryenergy}
\mathcal{Q}^\alpha(\varphi):=\int_\Gamma\int_\Gamma \frac{(\varphi(x)-\varphi(y))^2}{\sigma_\alpha(x,|x-y|)\mu(B(x,|x-y|))}\mu(dx)\mu(dy),\quad\varphi\in L^2(\Gamma,\mu).
\end{equation}
We follow \cite{Jonsson94} and write $B_\alpha^{2,2}(\Gamma)$ for the Hilbert space $(B_\alpha^{2,2}(\Gamma),\left\|\cdot\right\|_{B_\alpha^{2,2}(\Gamma)})$ of all $\varphi\in L^2(\Gamma,\mu)$ such that 
\begin{equation}\label{E:Besovnorm}
\left\|\varphi\right\|_{B_\alpha^{2,2}(\Gamma)}:=\left(\left\|\varphi\right\|_{L^2(\Gamma,\mu)}^2+ \mathcal{Q}^\alpha(\varphi)\right)^{1/2}
\end{equation}
is finite. Since $\mu$ satisfies conditions (\ref{E:refineddoubling}) and (\ref{E:lowerboundmu}), \cite[Theorem 1 and Proposition 2]{Jonsson94} show that $B_\alpha^{2,2}(\Gamma)$ is the trace space of $H^\alpha(\mathbb{R}^2)$ on $\Gamma$: Given $f\in H^\alpha(\mathbb{R}^2)$, the limit
\[\widetilde{f}(x):=\lim_{r\to 0}\fint_{B(x,r)} f(y)dy\]
of $f$ exists at $H^\alpha(\mathbb{R}^2)$-quasi every $x\in \mathbb{R}^2$. Setting 
\begin{equation}\label{E:traceasop}
\mathrm{Tr}_\Gamma f:=\widetilde{f}
\end{equation}
gives a bounded linear operator $\mathrm{Tr}_\Gamma: H^\alpha(\mathbb{R}^2)\to B_\alpha^{2,2}(\Gamma)$, and there is a bounded linear extension operator $\mathrm{E}_\Gamma:B_\alpha^{2,2}(\Gamma)\to H^\alpha(\mathbb{R}^2)$ of Whitney type such that $\mathrm{Tr}_\Gamma\circ \mathrm{E}_\Gamma$ is the identity. Since the space $\lip_c(\mathbb{R}^2)$ of compactly supported Lipschitz functions on $\mathbb{R}^2$ is dense in $H^\alpha(\mathbb{R}^2)$, the space $\lip(\Gamma)$ of Lipschitz functions on $\Gamma$ is dense in $B_{\alpha}^{2,2}(\Gamma)$. From $\mathcal{Q}^\alpha$ as defined in (\ref{E:boundaryenergy}), we obtain a symmetric bilinear (energy) form $(\mathcal{Q}^\alpha, B_{\alpha}^{2,2}(\Gamma))$ by polarization; it is a regular Dirichlet form on $L^2(\Gamma,\mu)$ in the sense of \cite{FOT94}.

\begin{remark}\label{R:JW}
If $\mu$ is $d$-regular, then \cite[Theorem 1]{Jonsson94} recovers a special case of \cite[Chapter VI, Theorem 2]{JonssonWallin84}, where $B_\alpha^{2,2}(\Gamma)$ (up to norm equivalence) is denoted by $B_{\alpha'}^{2,2}(\Gamma)$ with 
\begin{equation}\label{E:alphaprime}
\alpha'=\alpha-\frac{(2-d)}{2}. 
\end{equation}
See also \cite[Example 1]{Jonsson94}. In this case $0<\alpha'<1$ and $c^{-1}r^{2\alpha'}\leq \sigma_\alpha(x,r)\leq c\:r^{2\alpha'}$ for all $x\in \Gamma$ and $0<r\leq \diam \Gamma$. In this case of $d$-regular $\mu$ quadratic forms on $\Gamma$ were studied in \cite[Section 4.4]{FarkasJacob2001}.
\end{remark}

\begin{examples}
Suppose that $\Gamma$ is the classical Koch snowflake curve, $\mu$ is the natural arc-wise self-similar probability measure and $\sigma_\alpha(x,r)\equiv r^{2\alpha'}$ with $\alpha'$ as in (\ref{E:alphaprime}). Then 
\[\mathcal{Q}^\alpha(\varphi)=\int_\Gamma\int_\Gamma \frac{(\varphi(x)-\varphi(y))^2}{|x-y|^{2\alpha'}\mu(B(x,|x-y|))}\mu(dx)\mu(dy),\quad \varphi\in B_\alpha^{2,2}(\Gamma).\]
\end{examples}

Now let Assumption \ref{A:basicass} be in force and let 
\begin{equation}\label{E:constellation2}
\frac12<\alpha<1+\frac{2-s}{2};
\end{equation}
this implies (\ref{E:constellation}). Let $\sigma_\alpha$ be as in (\ref{E:constellation}) and (\ref{E:symbol}), and let 
\[A:= 32M^2S. \]
For each $n$ and $\xi\in \Gamma_n$ set
\[\varrho_n(\xi,r):=\begin{cases} \sigma_\alpha(\xi,r) & r> Ar_n,\\
r^{2\alpha-1} & r\leq  Ar_n,\end{cases}\]
and consider the quadratic form $\mathcal{Q}^\alpha_n$ on $L^2(\Gamma_n,\mu_n)$ defined by 
\begin{equation}\label{E:boundaryenergyn}
\mathcal{Q}^\alpha_n(\varphi):=\int_{\Gamma_n}\int_{\Gamma_n} \frac{(\varphi(\xi)-\varphi(\eta))^2}{\varrho_n(\xi,|\xi-\eta|)\mu_n(B(\xi,|\xi-\eta|))}\mu_n(d\xi)\mu_n(d\eta),\quad\varphi\in L^2(\Gamma_n,\mu_n).
\end{equation}
Since $1/2<\alpha<3/2$ by (\ref{E:constellation2}), the maximal domain $\{\varphi\in L^2(\Gamma_n,\mu_n): \mathcal{Q}^\alpha_n(\varphi)<+\infty\}$ of $\mathcal{Q}^\alpha_n$ agrees with the trace space $B_{\alpha}^{2,2}(\Gamma_n)$ of $H^{\alpha}(\mathbb{R}^2)$ on $\Gamma_n$ as discussed in the last section, and 
\[\varphi\mapsto \left(\left\|\varphi\right\|_{L^2(\Gamma_n,\mu_n)}^2+\mathcal{Q}^\alpha_n(\varphi)\right)^{1/2}\]
is a Hilbert space norm on $B_{\alpha}^{2,2}(\Gamma_n)$ equivalent to $\left\|\cdot\right\|_{B_{\alpha}^{2,2}(\Gamma_n)}$.
As before, polarization gives a symmetric bilinear (energy) form $(\mathcal{Q}^\alpha_n, B_{\alpha}^{2,2}(\Gamma_n))$; it is a regular Dirichlet form on $L^2(\Gamma_n,\mu_n)$ in the sense of \cite{FOT94}.

\begin{examples}
Suppose that $\Gamma$ is the classical Koch snowflake curve, the $\mathcal{I}_n$ are the natural partitions, $\mu$ is the natural arc-wise self-similar probability measure and $\sigma_\alpha(x,r)\equiv r^{2\alpha'}$ with $\alpha'$ as in (\ref{E:alphaprime}). Then there is some $c>0$ such that $r_n=c\:3^{-n}$, and we find that 
\begin{multline}
\mathcal{Q}^\alpha_n(\varphi)=\frac{3^n}{4^n}\int_{\Gamma_n}\int_{\Gamma_n\cap B(\xi,Ac\: 3^{-n})^c}\frac{(\varphi(\xi)-\varphi(\eta))^2}{|\xi-\eta|^{2\alpha+d-1}}\mathcal{H}^1(d\xi)\mathcal{H}^1(d\eta)\notag\\
+\frac{3^n}{4^n}\int_{\Gamma_n}\int_{\Gamma_n\cap B(\xi,Ac\: 3^{-n})}\frac{(\varphi(\xi)-\varphi(\eta))^2}{|\xi-\eta|^{2\alpha}}\mathcal{H}^1(d\xi)\mathcal{H}^1(d\eta).
\end{multline}
 
\end{examples}

\section{Mosco convergence of boundary energy forms}\label{S:Moscoboundary}

We make the following additional assumption.

\begin{assumption}\label{A:boundaryless}
For all $x\in \Gamma$ and all $0<r<\diam \Gamma$ we have $\mu(\partial B(x,r))=0$.
\end{assumption}

\begin{remark}
If $d>1$, then by upper $d$-regularity and density comparison, \cite[Theorem 6.9]{Mattila95}, Assumption \ref{A:boundaryless} is always satisfied. 
\end{remark}

In this section we prove the following convergence in the KS-generalized Mosco sense, see Appendix \ref{S:Notions} for the notion.

\begin{theorem}\label{T:Moscoboundary}
Let Assumptions \ref{A:basicass} and \ref{A:boundaryless} be satisfied, let $\alpha$ be as in (\ref{E:constellation2}) and let
$\mathcal{Q}^\alpha$ and $\mathcal{Q}^\alpha_n$ be as in (\ref{E:boundaryenergy}) and (\ref{E:boundaryenergyn}) respectively. Then 
\[\lim_{n\to\infty} \mathcal{Q}^\alpha_n=\mathcal{Q}^\alpha\] in the KS-generalized Mosco sense with respect to the convergence of Hilbert spaces in Lemma \ref{L:KSconvbd}.
\end{theorem}

One ingredient for the proof of Theorem \ref{T:Moscoboundary} is the following result on the pointwise
convergence of energy forms on Lipschitz functions. Given a function $u\in \lip_b(\mathbb{R}^2)$, we use the abbreviation $\mathcal{Q}^\alpha(u):=\mathcal{Q}^\alpha(u|_\Gamma)$, and similarly for the forms $\mathcal{Q}^\alpha_n$.

\begin{proposition}\label{P:convLip}
Let Assumptions \ref{A:basicass} and \ref{A:boundaryless} be satisfied, let $\alpha$ be as in (\ref{E:constellation2}) and let
$\mathcal{Q}^\alpha$ and $\mathcal{Q}^\alpha_n$ be as in (\ref{E:boundaryenergy}) and (\ref{E:boundaryenergyn}) respectively. Then $\lim_{n\to \infty} \mathcal{Q}^\alpha_n(u)=\mathcal{Q}^\alpha(u)$ for all $u\in \lip_b(\mathbb{R}^2)$.
\end{proposition}

To prove Proposition \ref{P:convLip}, note that  $\mathcal{Q}_n^\alpha(u)=\mathcal{Q}_n^{\alpha,s}(u)+\mathcal{Q}^{\alpha,\ell}(u)$, $u\in \lip_b(\mathbb{R}^2)$, where
\[\mathcal{Q}^{\alpha,s}_n(u):=\int_{\Gamma_n}\int_{\Gamma_n}\mathbf{1}_{\{|\xi-\eta|\leq  Ar_n\}}\frac{(u(\xi)-u(\eta))^2}{|\xi-\eta|^{2\alpha-1}\mu_n(B(\xi,|\xi-\eta|))}\mu_n(d\xi)\mu_n(d\eta).\]
and 
\[\mathcal{Q}_n^{\alpha,\ell}(u):=\int_{\Gamma_n}\int_{\Gamma_n}\mathbf{1}_{\{|\xi-\eta|> Ar_n\}}\frac{(u(\xi)-u(\eta))^2}{\sigma_\alpha(\xi,|\xi-\eta|))\mu_n(B(\xi,|\xi-\eta|))}\mu_n(d\xi)\mu_n(d\eta).\]
We claim that under the hypotheses of Proposition \ref{P:convLip},
\begin{equation}\label{E:claimcores}
\lim_{n\to \infty}\mathcal{Q}_n^{\alpha,s}(u)=0
\end{equation}
and 
\begin{equation}\label{E:claimcorel}
\lim_{n\to \infty}\mathcal{Q}_n^{\alpha,\ell}(u)=\mathcal{Q}^\alpha(u)
\end{equation}
for any $u\in \lip_b(\mathbb{R}^2)$. If so, then Proposition \ref{P:convLip} follows. We first prove claim (\ref{E:claimcores}).

\begin{proof}[Proof of (\ref{E:claimcores}).]
Using the Lipschitz property of $u$ and Lemma \ref{L:mundoubling} with $C=2A/M$, 
\begin{align}
\mathcal{Q}^{\alpha,s}_n(u)&\leq \lip(u)^2\int_{\Gamma_n}\int_{\overline{B(\eta,Ar_n)}}\frac{|\xi-\eta|^{3-2\alpha}}{\mu_n(B(\xi,|\xi-\eta|))}\mu_n(d\xi)\mu_n(d\eta)\notag\\
&\leq c_{D}\:\lip(u)^2\int_{\Gamma_n}\int_{\overline{B(\eta,Ar_n)}}\frac{|\xi-\eta|^{3-2\alpha}}{\mu_n(B(\eta,|\xi-\eta|))}\mu_n(d\xi)\mu_n(d\eta).\notag
\end{align}
For fixed $\eta\in \Gamma_n$ we rewrite the inner integral as a Stieltjes integral with respect to $r\mapsto \mu_n(B(\eta,r))$; and since $2\alpha<3$ by (\ref{E:sd}) and (\ref{E:constellation2}), this gives 
\begin{align}
\int_0^{Ar_n}\frac{r^{3-2\alpha}}{\mu_n(B(\eta,r))}d\mu_n(B(\eta,r))&=\sum_{j=0}^\infty \int_{2^{-j-1}Ar_n}^{2^{-j}Ar_n}\frac{r^{3-2\alpha}}{\mu_n(B(\eta,r))}d\mu_n(B(\eta,r))\notag\\
&\leq (A r_n)^{3-2\alpha} \sum_{j=0}^\infty 2^{-j(3-2\alpha)}\frac{\mu(B(\eta,2^{-j}Ar_n))}{\mu(B(\eta,2^{-j-1}Ar_n))}\notag\\
&\leq \frac{c_{D}(A r_n)^{3-2\alpha}}{1-2^{3-2\alpha}}.\notag
\end{align}
Since this goes to zero uniformly in $\eta$ and $\mu_n$ is a probability measure, (\ref{E:claimcores}) follows.
\end{proof}

The proof of (\ref{E:claimcorel}) uses two observations. The first is the following comparison lemma. 
\begin{lemma}\label{L:compare} Let Assumption \ref{A:basicass} be satisfied.
Suppose that $n\geq 1$, $x,y\in \Gamma$ with $|x-y|>\frac{A}{2}r_n$ and $w\in e_n(x)$,  $z\in e_n(y)$, where $r_n$ is as in (\ref{E:rn}).
\begin{enumerate}
\item[(i)] We have 
\begin{equation}\label{E:difference}
||x-y|-|w-z||\leq \frac{A}{8} r_n.
\end{equation}
\item[(ii)] For any $\alpha$ as in (\ref{E:constellation}) we have 
\begin{equation}\label{E:sigmacompare}
c^{-1}\sigma_\alpha(B(x,|x-y|))\leq \sigma_\alpha(B(w,|w-z|)\leq c\:\sigma_\alpha(B(x,|x-y|))
\end{equation}
with $c:=c_{\mu,D}^2\max(2^{2\alpha-2},2^{2-2\alpha})$.
\item[(iii)] We have 
\begin{equation}\label{E:mumuncompare}
c_{\mu,D}^{-1}\mu(B(x,|x-y|))\leq \mu_n(B(w,|w-z|)\leq c_{\mu,D}\mu(B(x,|x-y|)).
\end{equation}
\end{enumerate}
\end{lemma}

\begin{proof}
To shorten notation, we write $C:=\frac{A}{16}$. By (\ref{E:basicass}) and (\ref{E:convex}) we have
\begin{equation}\label{E:convex2}
|x-w|\leq C r_n\quad \text{and}\quad |y-z|\leq C r_n
\end{equation}  
whenever $w\in e_n(x)$ and $z\in e_n(y)$, so that the triangle inequality gives (\ref{E:difference}). Since $x$ and $y$ are far apart, it follows that  
\begin{equation}\label{E:easy}
|x-y|/2\leq |w-z|\leq 2|x-y|
\end{equation}
and that $B(x,|x-y|/4)\subset B(w,|w-z|)$ and $B(w,|w-z|)\subset B(x,4|x-y|)$. Together with (\ref{E:mudoubling}) this gives (\ref{E:sigmacompare}). Using (\ref{E:defmun}) and (\ref{E:convex2}) we obtain 
\begin{multline}
\mu_n(B(w,|w-z|))=\int_\Gamma\fint_{e_n(x')}\mathbf{1}_{B(w',|w-z|)}(w)\mathcal{H}^1(dw')\mu(dx')\geq \int_\Gamma \mathbf{1}_{B(x',|w-z|-Cr_n)}(w)\mu(dx')\notag\\
=\mu(B(w,|w-z|-Cr_n))\geq \mu(B(x,|w-z|-2Cr_n))\notag\\
\geq \mu(B(x,|x-y|-4C r_n))\geq \mu(B(x,|x-y|/2)),
\end{multline} 
and by (\ref{E:mudoubling}) the lower bound in (\ref{E:mumuncompare}) follows. The upper bound is seen similarly.
\end{proof}

The second observation is a continuity lemma. Under Assumption \ref{A:boundaryless} we have 
\begin{equation}\label{E:openorcloset}
\mu(B(x,r))=\mu(\overline{B(x,r)}),\quad x\in \Gamma,\quad 0<r\leq \diam \Gamma.
\end{equation}
It is convenient to use the notation $\mu(B(x,0)):=0$, $x\in \Gamma$; then $(x,r)\mapsto \mu(B(x,r))$ may be seen as a function on $\Gamma\times [0,\diam\:\Gamma]$.

\begin{lemma}\label{L:cont}
Let Assumptions \ref{A:basicass} and \ref{A:boundaryless} be satisfied.
\begin{enumerate}
\item[(i)] The function $(x,r)\mapsto \mu(B(x,r))$ is uniformly continuous on $\Gamma\times [0,\diam\:\Gamma]$. 
\item[(ii)] If $\alpha$ is as in (\ref{E:constellation}), then we have 
\begin{equation}\label{E:sigmacont}
\lim_{n\to \infty}\sup_{w\in e_n(x), z\in e_n(y)}|\sigma_\alpha(w,|w-z|)-\sigma_\alpha(x,|x-y|)|=0,\quad (x,y)\in \Gamma\times \Gamma\setminus \diag,
\end{equation}
and for any $\delta>0$ this convergence is 
uniform on $\{(x,y)\in \Gamma\times\Gamma:\ |x-y|\geq\delta\}$.
\item[(iii)]  We have
\begin{equation}\label{E:mumuncont}
\lim_{n\to \infty}\sup_{w\in e_n(x), z\in e_n(y)}|\mu_n(B(w,|w-z|))-\mu(B(x,|x-y|))|=0,\quad (x,y)\in \Gamma\times \Gamma\setminus \diag,
\end{equation}
and for any $\delta>0$ this convergence is 
uniform on $\{(x,y)\in \Gamma\times\Gamma:\ |x-y|\geq\delta\}$.
\end{enumerate}
\end{lemma} 

\begin{proof}
Given $x\in \Gamma$ and $0<r\leq \diam \Gamma$ the continuity of $\mu$ from below gives
\[\mu(B(x,r))=\mu\left(\bigcup_{n=1}^\infty B(x,r_n)\right)=\lim_{n\to \infty} \mu(B(x,r_n))\]
for any sequence $(r_n)_n$ with $r_n\uparrow r$. Given $x\in \Gamma$ and $0\leq r\leq \diam \Gamma$ the
continuity of $\mu$ from above and (\ref{E:openorcloset}) give  
\[\mu(B(x,r))=\mu\left(\bigcap_{n=1}^\infty \overline{B(x,r_n)}\right)=\lim_{n\to \infty} \mu(\overline{B(x,r_n)})=\lim_{n\to \infty} \mu(B(x,r_n))\]
for any sequence $(r_n)_n\subset [0,\diam \Gamma]$ with $r_n\downarrow r$. Given $(x,r),(x',r')\in \Gamma\times [0,\diam \Gamma]$ with $|(x,r)-(x',r')|<\varepsilon$, we have $B(x',r')\subset B(x,r+2\varepsilon)$ and therefore 
\[\mu(B(x',r'))-\mu(B(x,r))\leq \mu(B(x,r+2\varepsilon))-\mu(B(x,r)).\]
For $r=0$ the left-hand side is nonnegative. For $r>0$ we may choose $\varepsilon$ small enough to have $B(x,r-2\varepsilon)\subset B(x',r')$ and therefore
\[\mu(B(x,r))-\mu(B(x',r'))\leq \mu(B(x,r))-\mu(B(x,r-2\varepsilon)).\]
Together with the preceding, this implies that $(x,r)\mapsto \mu(B(x,r))$ is continuous; the compactness of $\Gamma\times [0,\diam\:\Gamma]$ implies the uniform continuity as claimed in (i).

If $n$ is large enough, then for any $w\in e_n(x)$ and $z\in e_n(y)$ we have 
\begin{multline}
|\sigma_\alpha(w,|w-z|)-\sigma_\alpha(x,|x-y|)|\notag\\
\leq \sup_{w\in e_n(x), z\in e_n(y)}||w-z|^{2\alpha-2}-|x-y|^{2\alpha-2}|\mu(B(w,|w-z|))\notag\\
+|x-y|^{2\alpha-2}\sup_{w\in e_n(x), z\in e_n(y)}|\mu(B(w,|w-z|))-\mu(B(x,|x-y|))|.\notag
\end{multline}
The first summand goes to zero by (\ref{E:difference}), (\ref{E:convex2}) and the uniform continuity of $(w,z)\mapsto |w-z|^{2\alpha-2}$ outside a neighborhood of the diagonal. As in the proof of Lemma \ref{L:compare}, let $C:=\frac{A}{16}$. Since for large enough $n$ we have
\[B(x,|x-y|-3C r_n)\subset B(w,|w-z|)\subset B(x,|x-y|+3C r_n)\]
for all $w\in e_n(x)$ and $z\in e_n(y)$, the second summand goes to zero by (i). This gives the first claim in (ii). The second claim follows from the uniform continuity in (i), note that for large $n$ we have $|x-y|\geq \delta>3C r_n$. 

To see (iii), suppose that $(x,y)\in \Gamma\times\Gamma\setminus \diag$, $w\in e_n(x)$ and $z\in e_n(y)$. If $x'\in \Gamma\setminus V_\ast$ and $w'\in e_n(x')$, then 
\begin{align}
\big|\mathbf{1}_{B(w,\left|w-z\right|)}(w')&-\mathbf{1}_{B(x,\left|x-y\right|)}(x')\big|\notag\\
&=\left|\mathbf{1}_{B(w',\left|w-z\right|)}(w)-\mathbf{1}_{B(x,\left|x-y\right|)}(x')\right|\notag\\
&=\mathbf{1}_{B(w',\left|w-z\right|)}(w)\mathbf{1}_{B(x,\left|x-y\right|)^c}(x')+\mathbf{1}_{B(w',\left|w-z\right|)^c}(w)\mathbf{1}_{B(x,\left|x-y\right|)}(x').\notag
\end{align}
Since 
\[B(w',|w-z|)\subset B(w',|x-y|+2C r_n)\subset B(x',|x-y|+3C r_n)\]
by (\ref{E:difference}) and (\ref{E:convex2}) and similarly
\[B(x',|x-y|- 3C r_n)\subset B(w',|x-y|-2C r_n) \subset B(w',|w-z|),  \]
the preceding is bounded by 
\begin{align}
&\mathbf{1}_{B(x',\left|x-y\right|+ 3C r_n)}(w)\mathbf{1}_{B(x,\left|x-y\right|)^c}(x')+\mathbf{1}_{B(x',\left|x-y\right|-3C r_n))^c}(w)\mathbf{1}_{B(x,\left|x-y\right|))}(x')\notag\\
&=\mathbf{1}_{B(w,\left|x-y\right|+3C r_n)}(x')\mathbf{1}_{B(x,\left|x-y\right|)^c}(x')+\mathbf{1}_{B(w,\left|x-y\right|-3C r_n)^c}(x')\mathbf{1}_{B(x,\left|x-y\right|)}(x')\notag\\
&\leq \mathbf{1}_{B(x,\left|x-y\right|+4C r_n)}(x')\mathbf{1}_{B(x,\left|x-y\right|)^c}(x')+\mathbf{1}_{B(x,\left|x-y\right|-4C r_n)^c}(x')\mathbf{1}_{B(x,\left|x-y\right|)}(x').\notag
\end{align}
Therefore 
\begin{align}
\big|\mu_n(B(w,\left|w-z\right|))&-\mu(B(x,\left|x-y\right|))\big|\notag\\
&=\left|\int_{\Gamma_n}\mathbf{1}_{B(w,\left|w-z\right|)}(w')\mu_n(dw')-\int_\Gamma\mathbf{1}_{B(x,\left|x-y\right|)}(x')\mu(dx')\right|\notag\\
&\leq \int_\Gamma \fint_{e_n(x')}\left|\mathbf{1}_{B(w,\left|w-z\right|)}(w')-\mathbf{1}_{B(x,\left|x-y\right|)}(x')\right|\:\mathcal{H}^1(dw')\mu(dx')\notag\\
&\leq \int_\Gamma \mathbf{1}_{B(x,\left|x-y\right|+4C r_n)}(x')\mathbf{1}_{B(x,\left|x-y\right|)^c}(x')\mu(dx')\notag\\
&\hspace{80pt}+\int_K \mathbf{1}_{B(x,\left|x-y\right|-4C r_n)^c}(x')\mathbf{1}_{B(x,\left|x-y\right|)}(x')\mu(dx')\notag\\
&=\mu(B(x,\left|x-y\right|+4C r_n))-\mu(B(x,\left|x-y\right|))\notag\\
&\hspace{80pt}+\mu(B(x,\left|x-y\right|))-\mu(B(x,\left|x-y\right|-4C r_n))\notag\\
&=\mu(B(x,\left|x-y\right|+4C r_n))-\mu(B(x,\left|x-y\right|- 4C r_n)),\notag
\end{align}
and (iii) follows using (i).
\end{proof}

We can now prove claim (\ref{E:claimcorel}).

\begin{proof}[Proof of (\ref{E:claimcorel}).] 
Our first claim is that, given $u\in \lip_b(\mathbb{R}^2)$ and $(x,y)\in \Gamma\times\Gamma\setminus \diag$, we have 
\begin{multline}\label{E:firstclaim}
\lim_{n\to \infty}\mathbf{1}_{\{|x-y|>(A\pm 2C)r_n\}}\fint_{e_n(y)}\fint_{e_n(x)}\frac{(u(w)-u(z))^2}{\sigma_\alpha(w,|w-z|)\mu_n(B(w,|w-z|))}\mathcal{H}^1(dw)\mathcal{H}^1(dz)\\
=\frac{(u(x)-u(y))^2}{\sigma_\alpha(x,|x-y|)\mu(B(x,|x-y|))};
\end{multline}
here $C$ is again as in (\ref{E:difference}). Clearly the indicator becomes one for large enough $n$, and
\[\lim_{n\to \infty}\fint_{e_n(y)}\fint_{e_n(x)}(u(w)-u(z))^2\mathcal{H}^1(dw)\mathcal{H}^1(dz)=(u(x)-u(y))^2,\]
because $(w,z)\mapsto (u(w)-u(z))^2$ is uniformly continuous on $\mathbb{R}^2$. To verify (\ref{E:firstclaim}), it therefore suffices to show that 
\begin{multline}\label{E:suffclaim}
\lim_{n\to \infty} \fint_{e_n(y)}\fint_{e_n(x)}(u(w)-u(z))^2 \Big(\frac{1}{\sigma_\alpha(w,|w-z|)\mu_n(B(w,|w-z|))}\\
-\frac{1}{\sigma_\alpha(x,|x-y|)\mu(B(x,|x-y|))}\Big)\mathcal{H}^1(dw)\mathcal{H}^1(dz)=0
\end{multline}
for any fixed $(x,y)\in \Gamma\times\Gamma\setminus \diag$. This is a consequence of (\ref{E:sigmacont}) and (\ref{E:mumuncont}): We have 
\begin{multline}\label{E:diffofreciprocals}
\Big|\frac{1}{\sigma_\alpha(w,|w-z|)\mu_n(B(w,|w-z|))}-\frac{1}{\sigma_\alpha(x,|x-y|)\mu(B(x,|x-y|))}\Big|\\
\leq \Big|\frac{1}{\sigma_\alpha(w,|w-z|)}-\frac{1}{\sigma_\alpha(x,|x-y|)}\Big|\frac{1}{\mu_n(B(w,|w-z|))}\\
+\frac{1}{\sigma_\alpha(x,|x-y|)}\Big|\frac{1}{\mu_n(B(w,|w-z|))}-\frac{1}{\mu(B(x,|x-y|))}\Big|.
\end{multline}
By (\ref{E:sigmacompare}) and (\ref{E:mumuncompare}) the first summand on the right-hand side is bounded by 
\begin{equation}\label{E:sigmabound}
c\:\frac{|\sigma_\alpha(w,|w-z|)-\sigma_\alpha(x,|x-y|)|}{\sigma_\alpha(x,|x-y|)^2\mu(B((x,|x-y|))}
\end{equation}
and the second by 
\begin{equation}\label{E:mumunbound}
c\:\frac{|\mu(B(x,|x-y|))-\mu_n(B(w,|w-z|))|}{\sigma_\alpha(x,|x-y|)\mu(B((x,|x-y|))^2},
\end{equation}
here $c>1$ is a constant independent of $w$, $x$, $y$ and $z$. Since $x$ and $y$ are fixed, (\ref{E:suffclaim}) now follows using Lemma \ref{L:cont} (ii) and (iii).

By Lemma \ref{L:compare} (ii) and (iii) and (\ref{E:easy}) we can find a constant $c'>0$ such that for any $x,y\in \Gamma\times\Gamma\setminus \diag$, any $u\in \lip_b(\mathbb{R}^2)$ and any $n$ we have 
\begin{multline}\label{E:majorant}
\mathbf{1}_{\{|x-y|>(A\pm 2C)r_n\}}\fint_{e_n(y)}\fint_{e_n(x)}\frac{(u(w)-u(z))^2}{\sigma_\alpha(w,|w-z|)\mu_n(B(w,|w-z|))}\mathcal{H}^1(dw)\mathcal{H}^1(dz)\\
\leq \frac{c'\lip(u)^2|x-y|^2}{\sigma_\alpha(x,|x-y|)\mu(B(x,|x-y|))}.
\end{multline}
By (\ref{E:symbol}), (\ref{E:uplowreg}), (\ref{E:mudoubling}) and (\ref{E:constellation2}), we have 
\begin{multline}
\int_\Gamma \frac{|x-y|^2}{\sigma_\alpha(x,|x-y|)\mu(B(x,|x-y|))}\mu(dx)\leq c\: \int_\Gamma \frac{|x-y|^{4-2\alpha-s}}{\mu(B(y,|x-y|))}\mu(dx)\notag\\
=c\:\int_0^{\diam\Gamma}\frac{r^{4-2\alpha-s}}{\mu(B(y,r))}d\mu(B(y,r))\leq c\:\sum_{j=0}^\infty 2^{-j(4-2\alpha-s)}\frac{\mu(B(y,2^{-j}\diam\Gamma))}{\mu(B(y,2^{-j-1}\diam\Gamma))}\leq \frac{c\:c_{\mu,D}}{1-2^{4-2\alpha-s}}\notag
\end{multline}
for any $y\in \Gamma$ and with constants $c>0$ independent of $y$. Since $\mu$ is a probability measure, this shows that the right-hand side of (\ref{E:majorant}) is in $L^1(\Gamma\times\Gamma,\mu\otimes\mu)$, so that by (\ref{E:firstclaim}) and dominated convergence, $\mathcal{Q}^{\alpha}(u)$ is seen to equal
\[\lim_{n\to \infty} \int\int_{\{|x-y|>(A\pm 2C)r_n\}}\fint_{e_n(y)}\fint_{e_n(x)}\frac{(u(w)-u(z))^2}{\sigma_\alpha(w,|w-z|)\mu_n(B(w,|w-z|))}\mathcal{H}^1(dw)\mathcal{H}^1(dz)\mu(dx)\mu(dy).\]
Since by (\ref{E:convex2}) we have 
\begin{multline}
\{(x,y,w,z)\in \Gamma^2\times \Gamma_n^2: |x-y|>(A+2C)r_n, w\in e_n(x), z\in e_n(y)\}\notag\\
\subset \{(x,y,w,z)\in \Gamma^2\times \Gamma_n^2: w\in e_n(x), z\in e_n(y), |w-z|>Ar_n\}\notag\\
\subset \{(x,y,w,z)\in \Gamma^2\times \Gamma_n^2: |x-y|>(A-2C)r_n, w\in e_n(x), z\in e_n(y)\},
\end{multline} 
monotonicity implies that 
\begin{align}
\mathcal{Q}^{\alpha}(u)&=\lim_{n\to \infty} \int_\Gamma\int_\Gamma\fint_{e_n(y)}\fint_{e_n(x)}\frac{\mathbf{1}_{\{|w-z|>Ar_n\}}(u(w)-u(z))^2}{\sigma_\alpha(w,|w-z|)\mu_n(B(w,|w-z|))}\mathcal{H}^1(dw)\mathcal{H}^1(dz)\mu(dx)\mu(dy)\notag\\
&=\lim_{n\to \infty}\int_{\Gamma_n}\int_{\Gamma_n}\mathbf{1}_{\{|\xi-\eta|>Ar_n\}}\frac{(u(\xi)-u(\eta))^2}{\sigma_\alpha(\xi, |\xi-\eta|))\mu_n(B(\xi,|\xi-\eta|))}\mu_n(d\xi)\mu_n(d\eta)\notag\\
&=\lim_{n\to \infty}\mathcal{Q}^{\alpha,\ell}_n(u).\notag
\end{align}
\end{proof}

\begin{proposition}\label{P:limsupboundary}
Let Assumptions \ref{A:basicass} and \ref{A:boundaryless} be satisfied, let $\alpha$ be as in (\ref{E:constellation2}) and let
$\mathcal{Q}^\alpha$ and $\mathcal{Q}^\alpha_n$ be as in (\ref{E:boundaryenergy}) and (\ref{E:boundaryenergyn}) respectively.
For any $\varphi\in L^2(\Gamma,\mu)$ there is a sequence $(\varphi_n)_n$ of elements $\varphi_n\in L^2(\Gamma_n,\mu_n)$ that converges KS-strongly to $\varphi\in L^2(\Gamma,\mu)$ and satisfies 
\[\limsup_{n\to\infty} \mathcal{Q}^\alpha_n(\varphi_n) \leq \mathcal{Q}^\alpha(\varphi). \]
\end{proposition}

\begin{proof}
Let $\varphi\in L^2(\Gamma,\mu)$. We may assume that $\varphi\in B_\alpha^{2,2}(\Gamma)$, because otherwise $\mathcal{Q}^\alpha(\varphi)=+\infty$ and there is nothing to prove. By density we can find a sequence $(\psi_n)_n\subset \lip(\Gamma)$ such that 
\begin{equation}\label{E:L2conv}
\lim_{n\to\infty}\left\|\psi_n-\varphi\right\|_{L^2(\Gamma,\mu)}=0
\end{equation}
and 
\begin{equation}\label{E:energiesconv}
\lim_{n\to\infty}\mathcal{Q}^\alpha(\psi_n)=\mathcal{Q}^\alpha(\varphi).
\end{equation}
For each $n\geq 1$ the Whitney extension $E_\Gamma\psi_n$ is an element of $\lip_b(\mathbb{R}^2)$; we denote it again by $\psi_n$. Given $j\geq 1$ we can find some $n_j\geq 1$ such that 
\[\Big|\left\|\psi_j-\psi_i\right\|_{L^2(\Gamma_n,\mu_n)}-\left\|\psi_j-\psi_i\right\|_{L^2(\Gamma,\mu)}\Big|<2^{-j},\ i=1,2,...,j,\]
and 
\begin{equation}\label{E:energiesclose}
\Big|\mathcal{Q}_n^\alpha(\psi_j)-\mathcal{Q}^\alpha(\psi_j)\Big|<2^{-j};
\end{equation}
this follows from (\ref{E:weakconveval}) and Proposition \ref{P:convLip}. From (\ref{E:L2conv}) it follows that for all $k\geq 1$ we can find some $i_k$  such that $\|\psi_i-\psi_j\|_{L^2(\Gamma,\mu)}<2^{-k}$ if $i,j\geq i_k$. Combining, it follows that 
\[\left\|\psi_j-\psi_{i_k}\right\|_{L^2(\Gamma_n,\mu_n)}\leq \left\|\psi_j-\psi_{i_k}\right\|_{L^2(\Gamma,\mu)}+2^{-j}\leq 2^{-k}+2^{-j}\]
for all $k$, all $j>i_k$ and all $n\geq n_j$. Now set $\varphi_n:=\psi_j$ for all $n_j< n\leq n_{j+1}$ and $\widetilde{\chi}_k:=\psi_{i_k}$. Then 
\[\limsup_{n\to\infty}\left\|\varphi_n-\widetilde{\chi}_k\right\|_{L^2(\Gamma_n,\mu_n)}\leq 2^{-k}\]
for all $k\geq 1$. Since on the other hand we have $\lim_{k\to\infty}\left\|\varphi-\widetilde{\chi}_k\right\|_{L^2(\Gamma,\mu)}=0$ by (\ref{E:L2conv}), it follows that $\lim_{n\to\infty}\varphi_n=\varphi$ KS-strongly. Combining (\ref{E:energiesconv}) and (\ref{E:energiesclose}) gives 
\[\lim_{n\to \infty} \mathcal{Q}^\alpha_n(\varphi_n)=\mathcal{Q}^\alpha(\varphi).\]
\end{proof}

\begin{proposition}\label{P:liminfboundary}

Let Assumptions \ref{A:basicass} and \ref{A:boundaryless} be satisfied, let $\alpha$ be as in (\ref{E:constellation2}) and let
$\mathcal{Q}^\alpha$ and $\mathcal{Q}^\alpha_n$ be as in (\ref{E:boundaryenergy}) and (\ref{E:boundaryenergyn}) respectively. If $(\varphi_n)_n$ is a sequence of elements $\varphi_n\in L^2(\Gamma_n,\mu_n)$ that converges KS-weakly to $\varphi\in L^2(\Gamma,\mu)$, then we have 
\[\mathcal{Q}^\alpha(\varphi)\leq \liminf_{n\to\infty} \mathcal{Q}^\alpha_n(\varphi_n). \]
\end{proposition}

\begin{proof}
Let $\delta>0$. We will show that 
\[\mathcal{Q}^{\alpha,\delta}(\psi):=\int_\Gamma\int_\Gamma\mathbf{1}_{\{|x-y|\geq \delta\}}\frac{(\psi(x)-\psi(y))^2}{\sigma_\alpha(x,|x-y|)\mu(B(x,|x-y|))}\mu(dx)\mu(dy),\quad \psi\in L^2(\Gamma,\mu),\]
satisfies
\begin{equation}\label{E:Qdeltaclaim}
\mathcal{Q}^{\alpha,\delta}(\varphi)\leq \liminf_{n\to\infty} \mathcal{Q}^\alpha_n(\varphi_n).
\end{equation}
Since $\delta>0$ was arbitrary, this proves the lemma. To prepare the proof of (\ref{E:Qdeltaclaim}), let 
\begin{multline}
\mathcal{Q}^{\alpha,\delta}_n(\psi):=\int_{\Gamma}\int_{\Gamma}\mathbf{1}_{\{|x-y|\geq \delta\}}\fint_{e_n(y)}\fint_{e_n(x)}\mathbf{1}_{\{|w-z|>Ar_n\}}\frac{(\psi(w)-\psi(z))^2}{\sigma_\alpha(w,|w-z|)\mu_n(B(w,|w-z|))}\times\notag\\
\times\mathcal{H}^1(dw)\mathcal{H}^1(dz)\mu(dx)\mu(dy),\quad  \psi\in L^2(\Gamma_n,\mu_n);
\end{multline}
note that for large enough $n$ we have $\mathbf{1}_{\{|w-z|>Ar_n\}}\equiv 1$ in the inner integral. Since $\delta>0$, Lemma \ref{L:cont} (ii) and (iii) and the estimate in (\ref{E:diffofreciprocals}), (\ref{E:sigmabound}) and (\ref{E:mumunbound}) imply that given $\varepsilon>0$, 
\[\sup_{|x-y|\geq \delta}\ \sup_{w\in e_n(x), z\in e_n(y)}\Big|\frac{1}{\sigma_\alpha(w,|w-z|)\mu_n(B(w,|w-z|))}-\frac{1}{\sigma_\alpha(x,|x-y|)\mu(B(x,|x-y|))}\Big|<\varepsilon\]
whenever $n$ is large enough. Accordingly, we find that 
\begin{align}
&\mathcal{Q}^{\alpha,\delta}_n(\varphi_n) \notag\\
&\geq \int_\Gamma\int_\Gamma\mathbf{1}_{\{|x-y|\geq \delta\}}\left(\frac{1}{\sigma_\alpha(x,|x-y|)\mu(B(x,|x-y|))}-\varepsilon\right)\times\notag\\
&\hspace{150pt} \times\fint_{e_n(y)}\fint_{e_n(x)}(\varphi_n(w)-\varphi_n(z))^2\mathcal{H}^1(dw)\mathcal{H}^1(dz)\mu(dx)\mu(dy)\notag\\
&\geq \int_\Gamma\int_\Gamma \mathbf{1}_{\{|x-y|\geq \delta\}}\left(\frac{1}{\sigma_\alpha(x,|x-y|)\mu(B(x,|x-y|))}-\varepsilon\right)([\varphi_n]_n(x)-[\varphi_n]_n(y))^2\mu(dx)\mu(dy)\notag\\
&\geq \int_\Gamma\int_\Gamma\mathbf{1}_{\{|x-y|\geq \delta\}}\frac{([\varphi_n]_n(x)-[\varphi_n]_n(y))^2}{\sigma_\alpha(x,|x-y|)\mu(B(x,|x-y|))}\mu(dx)\mu(dy)\notag\\
& \hspace{200pt} -\varepsilon\int_\Gamma\int_\Gamma ([\varphi_n]_n(x)-[\varphi_n]_n(y))^2\mu(dx)\mu(dy)\notag\\
&\geq \mathcal{Q}^{\alpha,\delta}([\varphi_n]_n)-4\varepsilon\left\|[\varphi_n]_n\right\|_{L^2(\Gamma,\mu)}^2\notag
\end{align}
for any such $n$, note that Jensen's inequality and (\ref{E:deffn}) give
\begin{align}
\fint_{e_n(y)}\fint_{e_n(x)}(\varphi_n(w)-\varphi_n(z))^2\mathcal{H}^1(dw)\mathcal{H}^1(dz)&\geq \left(\fint_{e_n(y)}\fint_{e_n(x)}(\varphi_n(w)-\varphi_n(z))\mathcal{H}^1(dw)\mathcal{H}^1(dz)\right)^2\notag\\
&=([\varphi_n]_n(x)-[\varphi_n]_n(y))^2.\notag
\end{align}
In a similar manner we can see that 
\begin{equation}\label{E:supfinite}
\sup_n \left\|[\varphi_n]_n\right\|_{L^2(\Gamma,\mu)}\leq \sup_n\left\|\varphi_n\right\|_{L^2(\Gamma_n,\mu_n)}<+\infty, 
\end{equation}
the finiteness is due to the KS-weak convergence. Now let $(\varphi_{n_k})_k\subset (\varphi_n)_n$ be such that 
\[\lim_{k\to\infty}\mathcal{Q}^\alpha_{n_k}(\varphi_{n_k})=\liminf_{n\to\infty}\mathcal{Q}^\alpha_n(\varphi_n);\]
we may assume that this quantity is finite. By (\ref{E:supfinite}) there is a subsequence $(n_{k_j})_j$ of $(n_k)_k$ such that the sequence of averages $([\varphi_{n_{k_j}}]_{n_{k_j}})_j$ converges to some $\psi$ weakly in $L^2(\Gamma,\mu)$, and by the Banach-Saks theorem, \cite[Section 38]{RieszNagy56}, we may assume that its arithmetic means converge to $\psi$ strongly in $L^2(\Gamma,\mu)$. Now the KS-weak convergence implies that $\psi=\varphi$, because
\begin{align}
\left\langle\psi,\chi\right\rangle_{L^2(\Gamma,\mu)}&=\lim_{j\to\infty}\big\langle [\varphi_{n_{k_j}}]_{n_{k_j}},\chi\big\rangle_{L^2(\Gamma,\mu)}\notag\\
&=\lim_{j\to\infty}\int_\Gamma\fint_{e_{n_{k_j}}(x)}\varphi_{n_{k_j}}(w)\mathcal{H}^1(dw)E_\Gamma\chi(x)\mu(dx)\notag\\
&=\lim_{j\to\infty}\int_\Gamma\fint_{e_{n_{k_j}}(x)}\varphi_{n_{k_j}}(w)E_\Gamma\chi(w)\mathcal{H}^1(dw)\mu(dx)\notag\\
&=\lim_{j\to\infty}\big\langle \varphi_{n_{k_j}}, E_\Gamma\chi\big\rangle_{L^2(\Gamma_{n_{k_j}},\mu_{n_{k_j}})}\notag\\
&=\left\langle \varphi,\chi\right\rangle_{L^2(\Gamma,\mu)}\notag
\end{align}
for any $\chi\in \lip(\Gamma)$. Since the quadratic form $\mathcal{Q}^{\alpha,\delta}$ is continuous on $L^2(\Gamma,\mu)$, it follows that 
\begin{equation}\label{E:Cesaroconv}
\lim_{N\to\infty}\mathcal{Q}^{\alpha,\delta}\Big(\frac{1}{N}\sum_{j=1}^N [\varphi_{n_{k_j}}]_{n_{k_j}}-\varphi\Big)=0.
\end{equation}
Now 
\[\liminf_{n\to\infty}\mathcal{Q}^\alpha_n(\varphi_n)=\lim_{j\to\infty} \mathcal{Q}_{n_{k_j}}^\alpha(\varphi_{n_{k_j}})\geq \limsup_{j\to\infty}\mathcal{Q}^{\alpha,\delta}([\varphi_{n_{k_j}}]_{n_{k_j}})-4\varepsilon\sup_j\big\|[\varphi_{n_{k_j}}]_{n_{k_j}}\big\|_{L^2(\Gamma,\mu)}^2. \]
Since the left-hand side does not depend on $\varepsilon$, we arrive at
\begin{align}
\liminf_{n\to\infty}\mathcal{Q}^\alpha_n(\varphi_n)^{1/2}&\geq \limsup_{j\to\infty}\mathcal{Q}^{\alpha,\delta}\big([\varphi_{n_{k_j}}]_{n_{k_j}}\big)^{1/2}\notag\\
&\geq \limsup_{j\to\infty}\frac{1}{N}\sum_{j=1}^N\mathcal{Q}^{\alpha,\delta}\big([\varphi_{n_{k_j}}]_{n_{k_j}}\big)^{1/2}\notag\\
&\geq \lim_{j\to\infty}\mathcal{Q}^{\alpha,\delta}\Big(\frac{1}{N}\sum_{j=1}^N [\varphi_{n_{k_j}}]_{n_{k_j}}\Big)^{1/2}\notag\\
&=\mathcal{Q}^{\alpha,\delta}(\varphi)^{1/2};\notag
\end{align}
here we have used the domination of the limsup over the Ces\`aro limsup, the triangle inequality for the seminorm $(\mathcal{Q}^{\alpha,\delta})^{1/2}$ and (\ref{E:Cesaroconv}). This is (\ref{E:Qdeltaclaim}).
\end{proof}

Propositions \ref{P:limsupboundary} and \ref{P:liminfboundary} now give Theorem \ref{T:Moscoboundary}.

\section{Approximation of quasidiscs}\label{S:Disks}

For homogeneous snowflake-like curves $\Gamma\subset \mathbb{R}^2$ the natural partitions $\mathcal{I}_n$ lead to polygonal curves $\Gamma_n$ that are Jordan curves. For general quasicircles $\Gamma$ and partitions we make this an assumption on the partitions $\mathcal{I}_n$.

\begin{assumption}\label{A:Jordan}
For any large enough $n$ the polygonal curve $\Gamma_n$ is a Jordan curve.
\end{assumption}

Given an arbitrary quasicircle $\Gamma\subset \mathbb{R}^2$, we can use (\ref{E:theta}) to actually construct finite partitions $\mathcal{I}_n$ satisfying Assumption \ref{A:Jordan} by an ad hoc procedure. We do not claim any kind of optimality.

\begin{lemma}\label{L:adhoc}
Let $0<p,q<1$ and $\omega=(\omega_1,\omega_2,...)\in \{p,q\}^\mathbb{N}$. Then, for any large enough $n$ there is a finite partition $\mathcal{I}_n$ of $\Gamma$ satisfying Assumption \ref{A:basicass} with $M=5$ and Assumption \ref{A:Jordan}. Moreover, for any large enough $n$ and any $\xi\in \Gamma_n$ and $r<\omega_1\cdots \omega_n\diam\Gamma$ the set $\Gamma_n\cap \overline{B(\xi,r)}$ is connected.
\end{lemma}

\begin{proof}
We consider $\Gamma$ oriented clockwise. We use $\theta$ as in (\ref{E:theta}) and write 
\begin{equation}\label{E:rnprime}
r_n':=4\omega_1\cdots \omega_n\diam\Gamma.
\end{equation} 
The idea of the proof is to construct a finite number of points $x_0,...,x_k\in \Gamma$ which will serve as the initial and terminal points of the arcs $a(x_j,x_{j+1})$ in the partition $\mathcal{I}_n$.

Choose some $x_0\in \Gamma$ and let $n$ be large enough to have $\Gamma\setminus B(x_0,r_n'/\theta)\neq \emptyset$. Let $x_1\in \partial B(x_0,r_n')$ be the point where $\Gamma$ exits $B(x_0,r_n')$ for the last time before exiting $\overline{B(x_0,r_n'/\theta)}$. Let $e_1$ be the line segment from $x_0$ to $x_1$.

Now suppose that $x_j\in \Gamma$ has been determined. If $x_j$ is visited by $\Gamma$ before its re-entry into $B(x_0,r'_n)$, let $x_{j+1}$ be the point on $\partial B(x_j,r_n')$ where $\Gamma$ exits $B(x_j,r_n')$ the last time before leaving $\overline{B(x_j,r_n'/\theta)}$. If $x_{j+1}$ is not in $B(x_0,r'_n)$, then let $e_{j+1}$ be the line segment from $x_j$ to $x_{j+1}$ and repeat this step with $x_{j+1}$ in place of $x_j$. After a finite number of steps we obtain a point $x_k\in \Gamma\cap B(x_0,r'_n)$. 

By (\ref{E:theta}) we have $x_{j+1}\notin \bigcup_{i=0}^j B(x_i,r_n')$ for all $j=0,...,k-2$, and we still have $x_k\notin \bigcup_{i=1}^{k-1} B(x_i,r_n')$. We have $|x_{j+1}-x_j|=r_n'$, $j=0,...,k-1$; therefore the corresponding arcs $a(x_j,x_{j+1})$ all satisfy 
\[ r_n'\leq \diam a(x_j,x_{j+1})\leq S\:r_n'\]
and the resulting line segments $e_{j+1}$, $j=0,...,k-1$, have length $r_n'$. Two neighboring line segments share an endpoint, but no other point. Line segments that do not share an endpoint cannot intersect: To intersect a line segment $e_{i+1}$ with $i<j-1$, the segment $e_{j+1}$ would have to cross the union $B(x_i,r_n')\cup B(x_{i+1},r_n')$, but this would require $e_{j+1}$ to have length at least $\sqrt{3}r_n'>r_n'$. 

It remains to discuss the situation of $x_k\in \Gamma\cap B(x_0,r'_n)$. Although some line segment $e_{i+1}$ with $i\leq k-2$ may intersect $B(x_0,r'_n)$, the point $x_k$ could never lie in the closure of the resulting disk segment cut off by $e_{i+1}$: Otherwise the maximal possible distance between $x_k$ and $e_{i+1}$ would be $r_n'-\frac12\sqrt{4r_n'^2-r_n'^2}=(1-\sqrt{3}/2)r_n'$, and since this is less than $(\sqrt{3}/2)r_n'$, the point $x_k$ would have to be an element of $B(x_i,r_n')\cup B(x_{i+1},r_n')$, which we know is impossible. Now recall that $x_k\notin B(x_1,r_n')$. In the case that $x_k\in B(x_0,r'_n)\setminus B(x_0,\frac14 r_n')$ we define $e_{k+1}$ to be the line segment from $x_k$ to $x_0$; it has length $\frac14 r_n'<\mathcal{H}^1(e_{k+1})< r_n'$, and by the preceding it cannot cross any other segment. This gives the Jordan curve $\Gamma_n:=e_1\cup e_2\cup ...\cup e_{k+1}$. In the case that $x_k\in B(x_0,\frac14 r'_n)$ we discard $x_0$ and redefine $e_1$ to be the line segment from $x_k$ to $x_1$; its length is $r_n'\leq \mathcal{H}^1(e_1)<\frac54 r_n'$, and again it cannot cross any other segment. This gives the Jordan curve $\Gamma_n:=e_1\cup e_2\cup ...\cup e_{k}$.

The last statement is immediate from the construction.
\end{proof}

Recall that the interior (bounded) planar domain enclosed by a Jordan curve is called a \emph{Jordan domain} and a domain bounded by a quasicircle is called a \emph{quasidisc}, \cite{Gehring82, Lehto87}. We write $\Omega$ to denote the interior (bounded) quasidisc bounded by $\Gamma$ and, with Assumptions \ref{A:basicass} and \ref{A:Jordan} in force, $\Omega_n$ to denote the Jordan domains bounded by the Jordan curves $\Gamma_n$, respectively. The following approximation result is easily seen.

\begin{lemma}\label{L:convcharfcts}
Let Assumptions \ref{A:basicass} and \ref{A:Jordan} be satisfied. Then we have $\lim_{n\to\infty} \mathbf{1}_{\Omega_n}=\mathbf{1}_\Omega$ in $L^p(\mathbb{R}^2)$, $p\in [1,+\infty)$.
\end{lemma}

\begin{proof}
Given $\varepsilon>0$ we have $\Gamma\subset (\Gamma_n)_\varepsilon$ and $\Gamma_n\subset (\Gamma)_\varepsilon$ for all sufficiently large $n$ by Lemma \ref{L:Hausdorffconvbd}. For such $n$ the set $\overline{\Omega}_n=\Omega_n\cup \Gamma_n$ is contained in $(\overline{\Omega})_\varepsilon$, therefore $\Omega_n\setminus \Omega\subset (\Gamma)_\varepsilon$ and 
$\mathcal{L}^2(\Omega_n\setminus \Omega)\leq\mathcal{L}^2((\Gamma)_\varepsilon)$.
The bounded domain $(\Omega^c \cup (\Gamma)_\varepsilon)^c$ enclosed by the \enquote{inner} boundary $\partial(\Gamma)_\varepsilon\cap \Omega$ of $(\Gamma)_\varepsilon$ is contained in $\Omega_n$, consequently $\Omega\setminus \Omega_n\subset (\Gamma)_\varepsilon$, and similarly as before we obtain $\mathcal{L}^2(\Omega\setminus \Omega_n)\leq \mathcal{L}^2((\Gamma)_\varepsilon)$.
Since $\mathcal{L}^2(\Gamma)=0$ by (\ref{E:sd}), the value $\mathcal{L}^2((\Gamma)_\varepsilon)$ can be made arbitrarily small. This means that 
\[\lim_{n\to\infty} \mathcal{L}^2((\Omega_n\setminus \Omega)\cup (\Omega\setminus \Omega_n))=0.\]
\end{proof}

For later purposes it is desirable to have some uniform control on the geometry of the approximating Jordan curves $\Gamma_n$.
At larger scales their behaviour with respect to condition (\ref{E:theta}) is uniformly controlled by that of $\Gamma$.
\begin{lemma}\label{L:controlatlargscales}
Let Assumptions \ref{A:basicass} and \ref{A:Jordan} be in force. Let $M$ be as in (\ref{E:basicass}), $\theta$ as in (\ref{E:theta}) and $r_n'$ as in (\ref{E:rnprime}). For any $n$, any $x\in \Gamma_n$ and any $r\geq \frac{M}{\theta}r_n'$ 
only the connected component of $\Gamma_n\cap \overline{B(x,\frac32r)}$ that contains $x$ intersects $B(x,\frac{\theta}{4}r)$.
\end{lemma}

\begin{proof}
Assume that there is a connected component of $\Gamma_n\cap \overline{B(x,\frac32 r)}$ that does not contain $x$ but intersects $B(x,\frac{\theta}{4} r)$. Then there are $y\in \Gamma_n\setminus \overline{B(x,\frac32 r)}$ and $z\in \Gamma_n\cap B(x,\frac{\theta}{4} r)$ such that, following $\Gamma_n$ in a fixed orientation, $y$ is visited between visiting $x$ and $z$. 
By Lemma \ref{L:Hausdorffconvbd} (i) we have $d_H(\Gamma_n,\Gamma)<\frac{M}{4}\:r_n'$; consequently there are $\xi,\eta,\zeta\in \Gamma$ such that $\max(|x-\xi|,|y-\eta|,|z-\zeta|)<\frac{M}{4} r_n'\leq \frac{\theta}{4} r$ and $\eta$ is visited by $\Gamma$ between visiting $\xi$ and $\zeta$. But then $|\xi-\zeta|<\theta r$ and since $\overline{B(\xi,r)}\subset B(x,(\frac32-\frac{\theta}{4})r)$ also $\eta\in \Gamma\setminus \overline{B(\xi,r)}$. This contradicts (\ref{E:theta}).
\end{proof}

For the polygonal Jordan curves $\Gamma_n$ constructed in Lemma \ref{L:adhoc} a uniform control at all scales is quickly seen. 
\begin{corollary}\label{C:fullcontrol}
Let $\theta$ be as in (\ref{E:theta}), $0<p,q<1$, $\omega=(\omega_1,\omega_2,...)\in \{p,q\}^\mathbb{N}$ and let $\mathcal{I}_n$ be the finite partitions of $\Gamma$ constructed in Lemma \ref{L:adhoc}. 
\begin{enumerate}
\item[(i)] For any large enough $n$, any $x\in \Gamma_n$ and any $r>0$ only the connected component of $\Gamma_n\cap \overline{B(x,r)}$ that contains $x$ intersects $B(x,\frac{\theta}{6}r)$.
\item[(ii)] There is a constant $S'=S'(\theta)\geq 1$, depending only on $\theta$, such that for any sufficiently large $n$ and any distinct $x,y\in \Gamma_n$ we have 
$\diam a_n(x,y)\leq S'|x-y|$, where $a_n(x,y)$ denotes a subarc of $\Gamma_n$ of minimal diameter connecting $x$ and $y$. In particular, each $\Gamma_n$ is a quasicircle and each $\Omega_n$ a quasidisc.
\end{enumerate}
\end{corollary}

\begin{proof}
Let $r_n'$ be as in (\ref{E:rnprime}). For large enough $n$ and any $r<\frac14 r_n'$ the set $\Gamma_n\cap \overline{B(x,r)}$ is connected by Lemma \ref{L:adhoc}. The same is true for $\frac14 r_n'\leq r<\frac{15}{2\theta} r_n'$: If there would be a connected component $C_n$ of $\Gamma_n\cap \overline{B(x,r)}$ other than that through $x$, the set $C_n\cap \overline{B(x,\frac{\theta}{30}r)}$ would be a connected component of $\Gamma_n \cap \overline{B(x,\frac{\theta}{30}r)}$ not containing $x$, but since $\frac{\theta}{30} r<\frac14 r_n'$, Lemma \ref{L:adhoc} prevents this from happening. For $r\geq \frac{15}{2\theta} r_n'$ Lemma \ref{L:controlatlargscales} states that only the $x$-component of $\Gamma_n\cap \overline{B(x,r)}$ can hit $B(x,\frac{\theta}{6}r)$. This shows (i). Statement (ii) now follows from \cite[2.25. Lemma]{MartioSarvas79} and its proof.
\end{proof}

Given $\varepsilon>0$, a domain $\Omega\subset \mathbb{R}^2$ is said to be an \emph{$(\varepsilon,\infty)$-uniform domain} if for any $x,y\in \Omega$ there is a rectifiable arc $\gamma\subset \Omega$ of length $\ell(\gamma)$ connecting $x$ and $y$ such that $\ell(\gamma)\leq\frac{|x-y|}{\varepsilon}$ and $\dist(z,\partial\Omega)\geq \varepsilon\:\frac{|x-z||y-z|}{|x-y|}$ for all $z\in \gamma$, \cite{Jones81}. A number of different but equivalent definitions can be found in \cite{Vaisala88}. 

It is well-known that a Jordan domain is a quasidisc if and only if it is an $(\varepsilon,\infty)$-domain for some $\varepsilon>0$, see \cite[Theorem C]{Jones81} or \cite[2.33. Corollary]{MartioSarvas79}. Actually, both implications in this equivalence are quantitative, and we will use one of them in a quantitative fashion: Any bounded quasidisc $\Omega$  with boundary $\Gamma=\partial\Omega$ and $S$ as in (\ref{E:3point}) is an $(\varepsilon,\infty)$-domain with parameter $\varepsilon=\varepsilon(S)>0$ depending only on $S$. This can for instance be concluded by combining \cite[Theorem 6.6 and its proof]{Lehto87} with \cite[Theorems 6.2 and 6.3]{Lehto87} or \cite[2.15. Theorem]{MartioSarvas79}.

\begin{remark}
To see that any bounded $(\varepsilon,\infty)$-domain in $\mathbb{R}^2$ is a quasidisc with $S=S(\varepsilon)\geq 1$ in (\ref{E:3point}) depending only on $\varepsilon$, one can combine \cite[Theorems 6.4 and 6.5]{Lehto87}.
\end{remark}

The preceding discussion gives the following result on approximating polygonal domains.

\begin{theorem}\label{T:adhoc}
Let $\theta$ be as in (\ref{E:theta}), $0<p,q<1$, $\omega=(\omega_1,\omega_2,...)\in \{p,q\}^\mathbb{N}$. Let $\mathcal{I}_n$ be the finite partitions of $\Gamma$ constructed in Lemma \ref{L:adhoc} and for any large enough $n$, let $\Omega_n$ be the quasidisc enclosed by the polygonal Jordan curve $\Gamma_n$. There is some $\varepsilon=\varepsilon(\theta)>0$, depending only on $\theta$, such that for any large enough $n$ the quasidisc $\Omega_n$ is an $(\varepsilon,\infty)$-uniform domain.
\end{theorem}

\section{Mosco convergence of Dirichlet energy forms}\label{S:Mosco}

Suppose that Assumptions \ref{A:basicass} and \ref{A:Jordan} are satisfied. On the domains $\Omega$ and $\Omega_n$ we consider the classical Dirichlet (energy) forms
\[\mathcal{D}(u)=\int_\Omega |\nabla u|^2\:dx,\quad u\in H^1(\Omega),\] 
and
\[\mathcal{D}_n(u)=\int_{\Omega_n} |\nabla u|^2\:dx,\quad u\in H^1(\Omega_n).\]
A priori they are densely defined quadratic forms on $L^2(\Omega)$ and $L^2(\Omega_n)$, respectively. However, it is convenient to view $H^1(\Omega)$ and $H^1(\Omega_n)$ as dense subspaces of $L^2(\mathbb{R}^2)$ by saying that $u\in L^2(\mathbb{R}^2)$ 
is an element of $H^1 (\Omega)$ if $u|_\Omega$ is, and similarly for $\Omega_n$. With this agreement in mind we extend the definitions of $\mathcal{D}$ and $\mathcal{D}_n$ to all of $u\in L^2(\mathbb{R}^2)$ by setting $\mathcal{D}(u):=+\infty$ for $u\in L^2(\mathbb{R}^2)$ such that $u|_\Omega\notin H^1(\Omega)$ and $\mathcal{D}_n(u):=+\infty$ for $u\in L^2(\mathbb{R}^2)$ such that $u|_{\Omega_n}\notin H^1(\Omega_n)$.

We now assume the following.  
\begin{assumption}\label{A:uniformeps}
There is some $\varepsilon>0$ such that for large any enough $n$, the domain $\Omega_n$ is an $(\varepsilon,\infty)$-uniform domain.
\end{assumption}

Note that Assumption \ref{A:uniformeps} is satisfied for the approximating quasidiscs $\Omega_n$ in Theorem \ref{T:adhoc}. Under Assumption \ref{A:uniformeps} the convergence of the Dirichlet forms can be seen using standard arguments;  the result and its proof are similar to earlier results in \cite{CapitanelliVivaldi2011, HR-PT21, HR-PT23, LanciaVernole14}.

\begin{theorem}\label{T:Mosco}
Let Assumptions \ref{A:basicass}, \ref{A:Jordan} and \ref{A:uniformeps} be satisfied. Then we have 
\[\lim_{n\to\infty} \mathcal{D}_n=\mathcal{D}\]
in the Mosco sense on $L^2(\mathbb{R}^2)$. 
\end{theorem}

We give a proof for the convenience of the reader. Since $\Omega$ is an $(\varepsilon,\infty)$-domain, there is a bounded linear extension operator $E_\Omega:H^1(\Omega)\to H^1(\mathbb{R}^2)$ by \cite[Theorem 1]{Jones81}. The same theorem together with  Assumption \ref{A:uniformeps} ensures that for all large enough $n$ there are linear extension operators $E_{\Omega_n}:H^1(\Omega_n)\to H^1(\mathbb{R}^2)$ whose operator norms are uniformly bounded with respect to $n$.

\begin{proof}
To prove the $\liminf$-condition, suppose that $(u_n)_n\subset L^2(\mathbb{R}^2)$ and $u\in L^2(\mathbb{R}^2)$ are such that $\lim_{n\to\infty} u_n=u$ weakly in $L^2(\mathbb{R}^2)$. We may assume that $L:=\liminf_{n\to\infty} \mathcal{D}_n(u_n)<+\infty$ and can find a sequence $(n_k)_k$ with $n_k\uparrow +\infty$ such that $\lim_{k\to\infty} \mathcal{D}_{n_k}(u_{n_k})=L$. Since $\sup_n\|u_n\|_{L^2(\mathbb{R}^2)}<+\infty$ by weak convergence, we find that $\sup_k\|u_{n_k}\|_{H^1(\Omega_{n_k})}<+\infty$, so that $\sup_k\|E_{\Omega_{n_k}}u_{n_k}\|_{H^1(\mathbb{R}^2)}<+\infty$ by \cite[Theorem 1]{Jones81}. Passing to subsequences if necessary, we may assume that $(E_{\Omega_{n_k}}u_{n_k})_k$ converges to some $u^\ast$ weakly in $H^1(\mathbb{R}^2)$ with convex combinations converging strongly to $u^\ast$ in $H^1(\mathbb{R}^2)$, \cite[Section 38]{RieszNagy56}. This implies that $u^\ast=u$. We may also assume that $(\nabla E_{\Omega_{n_k}}u_{n_k})_k$ converges weakly in $L^2(\mathbb{R}^2,\mathbb{R}^2)$ with convex combinations converging strongly; by the preceding its limit must be $\nabla u$. Using Lemma \ref{L:convcharfcts} we find that 
\[\lim_{k\to\infty} \mathbf{1}_{\Omega_{n_k}}\nabla u_{n_k}=\lim_{k\to\infty}\mathbf{1}_{\Omega_{n_k}}\nabla E_{\Omega_{n_k}}u_{n_k}=\mathbf{1}_\Omega \nabla u\]
weakly in $L^2(\mathbb{R}^2,\mathbb{R}^2)$ and accordingly, 
\[\mathcal{D}(u)=\int_\Omega |\nabla u|^2\:dx\leq \liminf_{k\to\infty}\int_{\Omega_{n_k}}|\nabla u_{n_k}|^2\:dx=\lim_{k\to\infty} \mathcal{D}_{n_k}(u_{n_k})=L.\]

To prove the $\limsup$-condition, suppose that $u\in L^2(\mathbb{R}^2)$. We may assume that $\mathcal{D}(u)<+\infty$. It follows that $u\in H^1(\Omega)$. Consequently $E_\Omega u$ is an element of $H^1(\mathbb{R}^2)$, and setting $u_n:=(E_\Omega u)|_{\Omega_n}$, we can use Lemma \ref{L:convcharfcts} and bounded convergence to see that  
\[\lim_{n\to\infty} u_n=u\quad \text{in $L^2(\mathbb{R}^2)$}\]
and 
\[\limsup_{n\to\infty}\mathcal{D}_n(u_n)=\lim_{n\to\infty}\int_D\mathbf{1}_{\Omega_n}|\nabla E_\Omega u|^2dx=\int_D\mathbf{1}_\Omega |\nabla E_\Omega u|^2dx=\mathcal{D}(u).\]
\end{proof}

\section{Mosco convergence of superpositions}\label{S:MoscoWentzell}

Consider the Borel measure
\[m:=\mathcal{L}^2|_\Omega+\mu.\]
The space $L^2(\overline{\Omega},m)$ is isometrically isomorphic to the orthogonal direct sum $L^2(\Omega)\oplus L^2(\Gamma,\mu)$ under the linear map $u\mapsto (u|_\Omega,u|_\Gamma)$, and we identify these spaces. In particular,
\begin{equation}\label{E:sps}
\left\langle u,v\right\rangle_{L^2(\overline{\Omega},m)}=\left\langle u|_\Omega,v|_\Omega\right\rangle_{L^2(\Omega)}+\left\langle u|_\Gamma,v|_\Gamma\right\rangle_{L^2(\Gamma,\mu)},\quad u,v\in L^2(\overline{\Omega},m).
\end{equation}
Now suppose that Assumptions \ref{A:basicass} and \ref{A:Jordan} are satisfied. Let $\Gamma_n$, $\mu_n$ and $\Omega_n$ be as before and set 
\[m_n:=\mathcal{L}^2|_{\Omega_n}+\mu_n.\]
Then analogous statements as above are true for each $L^2(\overline{\Omega}_n,m_n)$. 

Since $H^1(\Omega)$ is dense in $L^2(\Omega)$ and $\lip(\Gamma)$ is dense in $L^2(\Gamma,\mu)$, the direct sum $H^1(\Omega)\oplus \lip(\Gamma)$ is dense in $L^2(\overline{\Omega},m)$. Using Lemma \ref{L:KSconvbd}, Lemma \ref{L:convcharfcts}, bounded convergence and (\ref{E:sps}), we can conclude that the Hilbert spaces converge in the sense of \cite[Section 2.2]{KuwaeShioya03}, see Appendix \ref{S:Notions}.

\begin{lemma}\label{L:KSconv}  Let Assumptions \ref{A:basicass} and \ref{A:Jordan} be satisfied.
For any $u\in H^1(\Omega)\oplus \lip(\Gamma)$ we have 
\[\lim_{n\to \infty} \left\| (E_\Omega(u|_\Omega))|_{\Omega_n}+(E_\Gamma(u|_\Gamma))|_{\Gamma_n}\right\|_{L^2(\overline{\Omega}_n,m_n)}=\left\|u\right\|_{L^2(\overline{\Omega},m)}.\]
The sequence of Hilbert spaces $L^2(\overline{\Omega}_n,m_n)$ converges to $L^2(\overline{\Omega},m)$ with identification maps $u\mapsto (E_\Omega(u|_\Omega))|_{\Omega_n}+(E_\Gamma(u|_\Gamma))|_{\Gamma_n}$, $u\in H^1(\Omega)\oplus \lip(\Gamma)$.
\end{lemma}

Given an element $u$ of $L^2(\Omega)$ or $L^2(\Omega_n)$, we write $u^\circ$ for its continuation by zero to all of $\mathbb{R}^2$. The following will be used later on.

\begin{lemma}\label{L:inherit} Let Assumptions \ref{A:basicass} and \ref{A:Jordan} be satisfied.
\begin{enumerate}
\item[(i)] Let $v\in L^2(\mathbb{R}^2)$ and let $\varphi_n\in L^2(\Gamma_n,\mu_n)$ and $\varphi\in L^2(\Gamma,\mu)$ be such that $\lim_{n\to\infty} \varphi_n=\varphi$ KS-strongly with respect to the convergence of Hilbert spaces in Lemma \ref{L:KSconvbd}. Then $\lim_{n\to\infty} (v|_{\Omega_n}+\varphi_n)=v|_\Omega+\varphi$ KS-strongly with respect to the convergence of Hilbert spaces in Lemma \ref{L:KSconv}.
\item[(ii)] Let $u_n\in L^2(\overline{\Omega}_n,m_n)$ and $u\in L^2(\overline{\Omega},m)$ be such that $\lim_{n\to\infty} u_n=u$ KS-weakly with respect to the convergence of Hilbert spaces in Lemma \ref{L:KSconv}. Then $\lim_{n\to\infty} (u_n|_{\Omega_n})^\circ=(u|_\Omega)^\circ$ weakly in $L^2(\mathbb{R}^2)$ and $\lim_{n\to\infty} u|_{\Gamma_n}=u|_\Gamma$ KS-weakly with respect to the convergence of Hilbert spaces in Lemma \ref{L:KSconvbd}.
\end{enumerate}
\end{lemma}

\begin{proof}
Let $(\widetilde{\varphi}_m)_m\subset \lip(\Gamma)$ be such that 
\begin{equation}\label{E:bdapp}
\lim_{m\to\infty}\limsup_{n\to\infty}\|(E_\Gamma\widetilde{\varphi}_m)|_{\Gamma_n}-\varphi_n\|_{L^2(\Gamma_n,\mu_n)}=0\quad\text{and}\quad \lim_{m\to\infty} \|\widetilde{\varphi}_m-\varphi\|_{L^2(\Gamma,\mu)}=0
\end{equation}
and $(\widetilde{v}_m)_m\subset H^1(\Omega)$ such that 
\begin{equation}\label{E:bulkapp}
\lim_{m\to\infty}\|\widetilde{v}_m-v|_\Omega\|_{L^2(\Omega)}=0.
\end{equation}
Then for each $m$ we have $\widetilde{v}_m+\widetilde{\varphi}_m\in H^1(\Omega)\oplus \lip(\Gamma)$ and 
\[\lim_{m\to\infty} \|\widetilde{v}_m+\widetilde{\varphi}_m-(v|_\Omega+\varphi)\|_{L^2(\overline{\Omega},m)}=0\]
by (\ref{E:sps}). Since 
\[\|\mathbf{1}_{\Omega_n}(E_\Omega\widetilde{v}_m)-\mathbf{1}_{\Omega_n}v\|_{L^2(\mathbb{R}^2)}\leq \|(\mathbf{1}_{\Omega_n}-\mathbf{1}_\Omega)E_\Omega\widetilde{v}_m\|_{L^2(\mathbb{R}^2)}+\|(\mathbf{1}_{\Omega_n}-\mathbf{1}_{\Omega})v\|_{L^2(\mathbb{R}^2)}+\|\widetilde{v}_m-v|_\Omega\|_{L^2(\Omega)},\]
we can use (\ref{E:sps}), the subadditivity of $\limsup$, Lemma \ref{L:convcharfcts} and bounded convergence to see that 
\begin{align}
\limsup_{n\to\infty}\|(E_\Omega\widetilde{v}_m)|_{\Omega_n} &+(E_\Gamma\widetilde{\varphi}_m)|_{\Gamma_n}-(v|_{\Omega_n}+\varphi_n)\|_{L^2(\overline{\Omega}_n,m_n)}\notag\\
&\leq \limsup_{n\to\infty}\|(E_\Omega\widetilde{v}_m)|_{\Omega_n}-v|_{\Omega_n}\|_{L^2(\Omega_n)}+\limsup_{n\to\infty}\|(E_\Gamma\widetilde{\varphi}_m)|_{\Gamma_n}-\varphi_n\|_{L^2(\Gamma_n,\mu_n)}\notag\\
&\leq \|\widetilde{v}_m-v|_{\Omega}\|_{L^2(\Omega)}+\limsup_{n\to\infty}\|(E_\Gamma\widetilde{\varphi}_m)|_{\Gamma_n}-\varphi_n\|_{L^2(\Gamma_n,\mu_n)}.\notag
\end{align}
Using (\ref{E:bdapp}) and (\ref{E:bulkapp}) we arrive at 
\[\lim_{m\to\infty}\limsup_{n\to\infty}\|(E_\Omega\widetilde{v}_m)|_{\Omega_n}+(E_\Gamma\widetilde{\varphi}_m)|_{\Gamma_n}-(v|_{\Omega_n}+\varphi_n)\|_{L^2(\overline{\Omega}_n,m_n)}=0.\]
This shows (i). To see (ii), let $w\in L^2(\mathbb{R}^2)$. Then $\lim_{n\to\infty}\mathbf{1}_{\Omega_n}w=\mathbf{1}_\Omega w$ in $L^2(\mathbb{R}^2)$ by Lemma \ref{L:convcharfcts} and bounded convergence. By (i), applied with $\varphi_n= 0$ and $\varphi=0$, it follows that $\lim_{n\to\infty}w|_{\Omega_n}=w|_{\Omega}$ KS-strongly with respect to the convergence of Hilbert spaces in Lemma \ref{L:KSconv}. Consequently 
\[\lim_{n\to\infty}\big\langle (u_n|_{\Omega_n})^\circ,w\big\rangle_{L^2(\mathbb{R}^2)}=\lim_{n\to\infty}\big\langle u_n,w|_{\Omega_n} \big\rangle_{L^2(\Omega_n)}= \big\langle u,w|_{\Omega} \big\rangle_{L^2(\Omega)}= \big\langle (u|_{\Omega})^\circ,w\big\rangle_{L^2(\mathbb{R}^2)}.\]
This give the first claim in (ii). For the second, let $\varphi_n\in L^2(\Gamma_n,\mu_n)$ and $\varphi\in L^2(\Gamma,\mu)$ be such that $\lim_{n\to\infty} \varphi_n=\varphi$ KS-strongly with respect to the convergence of Hilbert spaces in Lemma \ref{L:KSconvbd}. Item (i), applied with $v=0$, gives the same convergence KS-strongly with respect to the convergence of Hilbert spaces in Lemma \ref{L:KSconv} and therefore 
\[\lim_{n\to\infty} \big\langle u_n|_{\Gamma_n},\varphi_n\big\rangle_{L^2(\Gamma_n,\mu_n)}=\big\langle u|_\Gamma,\varphi\big\rangle_{L^2(\Gamma,\mu)},\]
which shows the second claim.
\end{proof}

We choose $\alpha=1$ in (\ref{E:constellation}) and (\ref{E:constellation2}), note that by (\ref{E:sd}) this choice is admissible. For simplicity, we take $\sigma_1(x,r):=\mu(B(x,r))$ in (\ref{E:symbol}). Then 
\begin{equation}\label{E:Q1}
\mathcal{Q}^1(\varphi)=\int_\Gamma\int_\Gamma\frac{(\varphi(x)-\varphi(y))^2}{\mu(B(x,|x-y|))^2}\mu(dx)\mu(dy),\quad \varphi\in 
B_1^{2,2}(\Gamma),
\end{equation}
and 
\begin{multline}
\mathcal{Q}^1_n(\varphi)=\int_{\Gamma_n}\int_{\Gamma_n\cap B(\xi,A\:r_n)^c}\frac{(\varphi(\xi)-\varphi(\eta))^2}{\mu(B(\xi,|\xi-\eta|))\mu_n(B(\xi,|\xi-\eta|))}\mu_n(d\xi)\mu_n(d\eta)\notag\\
+\int_{\Gamma_n}\int_{\Gamma_n\cap B(\xi,A\:r_n)}\frac{(\varphi(\xi)-\varphi(\eta))^2}{|\xi-\eta|\mu_n(B(\xi,|\xi-\eta|))}\mu_n(d\xi)\mu_n(d\eta),\quad \varphi\in 
B_1^{2,2}(\Gamma_n).
\end{multline}

Composing $E_\Omega$ with (\ref{E:traceasop}) we obtain a bounded linear trace operator
\[\mathrm{Tr}_{\Omega,\Gamma}:=\mathrm{Tr}_{\Gamma}\circ E_\Omega:H^1(\Omega)\to B_1^{2,2}(\Gamma).\]
The space 
\[V(\overline{\Omega}):=\{v+\mathrm{Tr}_{\Omega,\Gamma}v:\ v\in H^1(\Omega)\}\] 
is a subspace of the vector space $H^1(\Omega)\oplus B_1^{2,2}(\Gamma)$; note that for an element $u=v+\mathrm{Tr}_{\Omega,\Gamma}v$ of $V(\overline{\Omega})$ we have $u|_\Omega=v$ and $u|_\Gamma=\mathrm{Tr}_{\Omega,\Gamma}v$. Since the space $V(\overline{\Omega})$ contains all restrictions of $\lip_c(\mathbb{R}^2)$-functions to $\overline{\Omega}$, it is dense in $L^2(\overline{\Omega},m)$ by Stone-Weierstrass. The space $\lip_c(\mathbb{R}^2)|_{\overline{\Omega}}$ of such restrictions to $\overline{\Omega}$ is a dense subspace of $V(\overline{\Omega})$; this is straightforward from the density of $\lip_c(\mathbb{R}^2)$ in $H^1(\mathbb{R}^2)$ and the continuity of the trace operator (\ref{E:traceasop}).

Now let Assumption \ref{A:uniformeps} be in force. Composition of the extension operators $E_{\Omega_n}$ with (\ref{E:traceasop}) gives linear trace operators 
\[ \mathrm{Tr}_{\Omega_n,\Gamma_n}:=\mathrm{Tr}_{\Gamma_n}\circ E_{\Omega_n}:H^1(\Omega_n)\to B_1^{2,2}(\Gamma_n).\]
Similarly as before, the space  $V(\overline{\Omega}_n):=\{(u,\mathrm{Tr}_{\Omega_n,\Gamma_n}u):\ u\in H^1(\Omega_n)\}$ is dense in $L^2(\overline{\Omega}_n,m_n)$.

We consider the quadratic forms 
\[\mathcal{E}(u):=\mathcal{D}(u|_\Omega)+\mathcal{Q}^1(u|_\Gamma),\quad u\in  V(\overline{\Omega}),\] 
and
\[\mathcal{E}_n(u):=\mathcal{D}_n(u|_{\Omega_n})+\mathcal{Q}^1_n(u|_{\Gamma_n}),\quad u\in  V(\overline{\Omega}_n);\]
they define Dirichlet forms $(\mathcal{E},V(\overline{\Omega}))$ and $(\mathcal{E}_n,V(\overline{\Omega}_n))$ on $L^2(\overline{\Omega},m)$ and $L^2(\overline{\Omega}_n,m_n)$, respectively. The Hilbert space norm $\|\cdot\|_{V(\overline{\Omega})}$ in $V(\overline{\Omega})$ is determined by
\[\|u\|_{V(\overline{\Omega})}^2= \mathcal{E}(u)+ \|u|_\Omega\|_{L^2(\Omega)}^2 +\|u|_{\Gamma}\|^{2}_{L^2(\Gamma,\mu)},\]
similarly for $V(\overline{\Omega}_n)$. We set 
\[\mathcal{E}(u):=+\infty\quad \text{for $u\in L^2(\overline{\Omega},m)\setminus V(\overline{\Omega})$}\quad \text{and }\quad \mathcal{E}_n(u):=+\infty\ \text{for $u\in L^2(\overline{\Omega}_n,m_n)\setminus V(\overline{\Omega}_n)$.}\]

We conclude the convergence of these Dirichlet forms in the KS-generalized Mosco sense.
\begin{theorem}\label{T:MoscoWentzell}
Let Assumptions \ref{A:basicass}, \ref{A:boundaryless}, \ref{A:Jordan} and \ref{A:uniformeps} be satisfied. Then we have 
\[\lim_{n\to\infty} \mathcal{E}_n=\mathcal{E}\]
in the KS-generalized Mosco sense with respect to the convergence of Hilbert spaces in Lemma \ref{L:KSconv}.
\end{theorem}

To verify Theorem \ref{T:MoscoWentzell} we use the following result, which is immediate from Proposition \ref{P:convLip}, Lemma \ref{L:convcharfcts} and bounded convergence.

\begin{proposition}\label{P:convLipWentzell}
Let Assumptions \ref{A:basicass}, \ref{A:boundaryless} and \ref{A:Jordan} be in force. Then we have 
\[\lim_{n\to \infty} \mathcal{E}_n(u)=\mathcal{E}(u),\quad u\in \lip_b(\mathbb{R}^2).\]
\end{proposition}

We prove Theorem \ref{T:MoscoWentzell}.

\begin{proof}
To see the $\liminf$-condition, suppose that $u_n\in L^2(\overline{\Omega}_n,m_n)$ and $u\in L^2(\overline{\Omega},m)$ be such that $\lim_{n\to\infty} u_n=u$ KS-weakly with respect to the convergence of Hilbert spaces in Lemma \ref{L:KSconv}. By Lemma \ref{L:inherit} (ii) we have
$\lim_{n\to\infty} (u_n|_{\Omega_n})^\circ=(u|_\Omega)^\circ$ weakly in $L^2(\mathbb{R}^2)$ and $\lim_{n\to\infty} u_n|_{\Gamma_n}=u|_\Gamma$ KS-weakly with respect to the convergence of Hilbert spaces in Lemma \ref{L:KSconvbd}. By Theorems \ref{T:Mosco} and \ref{T:Moscoboundary} therefore $\mathcal{D}(u|_\Omega)\leq \liminf_{n\to\infty} \mathcal{D}_n(u_n|_{\Omega_n})$ and $\mathcal{Q}^1(u|_\Gamma)\leq \liminf_{n\to\infty} \mathcal{Q}^1_n(u_n|_{\Gamma_n})$, and the superadditivity of $\liminf$ gives
\[\mathcal{E}(u)\leq \liminf_{n\to\infty} \mathcal{E}_n(u_n).\]
 
To see the $\limsup$-condition, suppose that $u\in L^2(\overline{\Omega},m)$. We may assume $\mathcal{E}(u)<+\infty$, so that 
$u\in V(\overline{\Omega})$. By density we can find a sequence $(v_m)_m\subset \lip_c(\mathbb{R}^2)|_{\overline{\Omega}}$ such that 
\[\lim_{m\to\infty} \mathcal{E}(v_m)=\mathcal{E}(u)\quad \text{and}\quad \lim_{m\to\infty}\|v_m-u\|_{L^2(\overline{\Omega},m)}=0.\]
One can now use Proposition \ref{P:convLipWentzell} and analogous arguments as in the proof of Proposition \ref{P:limsupboundary} to construct a sequence $(u_m)_m$ converging to $u$  KS-strongly with respect to the convergence of Hilbert spaces in Lemma \ref{L:KSconv} and satisfying 
\[\limsup_{n\to\infty} \mathcal{E}_n(u_n)= \mathcal{E}(u).\]
\end{proof}

\section{Application to elliptic and parabolic problems}\label{S:Apps}

To fix notation we write $(\mathcal{L}^1_\Gamma,\mathcal{D}(\mathcal{L}^1_\Gamma))$ for the infinitesimal generator of $(\mathcal{Q}^1,B^{2,2}_1(\Gamma,\mu))$ as in (\ref{E:Q1}), that is, the unique non-positive self-adjoint operator on $L^2(\Gamma,\mu)$ such that 
\[\mathcal{Q}^1(\varphi,\psi)=-\left\langle \mathcal{L}^1_\Gamma \varphi,\psi\right\rangle_{L^2(\Gamma,\mu)},\quad \varphi\in \mathcal{D}(\mathcal{L}^1_\Gamma),\ \psi\in B^{2,2}_1(\Gamma,\mu).\]
It has the representation 
\[\mathcal{L}_\Gamma^1\varphi(x)=\int_\Gamma(\varphi(y)-\varphi(x))\left\lbrace \mu(B(x,|x-y|))^{-2}+\mu(B(y,|x-y|))^{-2}\right\rbrace\mu(dy),\quad \varphi\in \mathcal{D}(\mathcal{L}^1_\Gamma).\]

The infinitesimal generator of $(\mathcal{D},H^1(\Omega))$ is the Neumann Laplacian $(\mathcal{L}_\Omega,\mathcal{D}(\mathcal{L}_\Omega))$ for $\Omega$, it is the unique non-positive self-adjoint operator on $L^2(\Omega)$ satisfying 
\[\mathcal{D}(u,v)=-\left\langle \mathcal{L}_\Omega u,v\right\rangle_{L^2(\Omega)},\quad u\in \mathcal{D}(\mathcal{L}_\Omega),\quad v\in H^1(\Omega).\]

Given $\lambda>0$, $f\in L^2(\Omega)$ and $\varphi\in L^2(\Gamma,\mu)$, we call $u\in V(\overline{\Omega})$ a \emph{weak solution} to the elliptic boundary value problem
\begin{equation}\label{E:elliptic}
\begin{cases} (\lambda-\mathcal{L}_\Omega)(u|_\Omega) =f & \text{on $\Omega$},\\
(\lambda-\mathcal{L}_\Gamma^1)(u|_\Gamma) =\varphi & \text{on $\Gamma$} \end{cases}
\end{equation}
if 
\[\mathcal{E}(u,v)+\lambda\int_{\overline{\Omega}} uv\:dm=\int_\Omega fv\:dx+\int_\Gamma \varphi v\:d\mu,\quad v\in V(\overline{\Omega}).\]

Under Assumptions \ref{A:basicass} and \ref{A:Jordan} we can also consider (\ref{E:elliptic}) with $\Omega_n$, $\Gamma_n$ and $\mu_n$ in place of $\Omega$, $\Gamma$ and $\mu$. The infinitesimal generator $(\mathcal{L}^1_{\Gamma_n},\mathcal{D}(\mathcal{L}^1_{\Gamma_n}))$ of $(\mathcal{Q}^1_n,B^{2,2}_1(\Gamma_n,\mu_n))$ is of the form 
\[\mathcal{L}^1_{\Gamma_n}\varphi(\xi)=\int_{\Gamma_n}(\varphi(\eta)-\varphi(\xi))j_n(\xi,\eta)\mu_n(d\eta),\quad \varphi\in \mathcal{D}(\mathcal{L}^1_{\Gamma_n}),\]
where
\begin{multline}
j_n(\xi,\eta)=\mathbf{1}_{B(\xi,Ar_n)^c}(\eta)\left\lbrace \mu(B(\xi,|\xi-\eta|))^{-1}\mu_n(B(\xi,|\xi-\eta|))^{-1}\right.\notag\\
\left.+\mu(B(\eta,|\xi-\eta|))^{-1}\mu_n(B(\eta,|\xi-\eta|))^{-1} \right\rbrace\notag\\
+\mathbf{1}_{B(\xi,Ar_n)}(\eta)|\xi-\eta|^{-1}\left\lbrace \mu_n(B(\xi,|\xi-\eta|))^{-1}+\mu_n(B(\eta,|\xi-\eta|))^{-1}\right\rbrace.
\end{multline}
The Riesz representation theorem gives the following standard result.
\begin{lemma}
Let $\lambda>0$. Then for any $f\in L^2(\Omega)$ and $\varphi\in L^2(\Gamma,\mu)$ there is a unique weak solution $u$ to (\ref{E:elliptic}). If Assumptions \ref{A:basicass} and \ref{A:Jordan} are satisfied, then for any large enough $n$ and any 
$f\in L^2(\Omega_n)$ and $\varphi\in L^2(\Gamma_n,\mu_n)$ there is a unique weak solution $u_n$ to the corresponding problem on $\Omega_n$.
\end{lemma}

\begin{remark}
The requirement to have $n$ large enough stems from Assumption \ref{A:Jordan}, which was formulated that way and makes sure the $\Gamma_n$ are boundaries of domains $\Omega_n$. 
\end{remark}

Theorem \ref{T:MoscoWentzell}, \cite[Theorem 2.4]{KuwaeShioya03}, Lemma \ref{L:inherit} (i) and Lemma \ref{L:traceconv} below give the next stability result.
\begin{theorem}\label{T:ellipticstability}
Suppose that Assumptions \ref{A:basicass}, \ref{A:boundaryless}, \ref{A:Jordan} and \ref{A:uniformeps} are in force. Let
$\lambda>0$, $f\in L^2(\mathbb{R}^2)$ and $g\in \lip_b(\mathbb{R}^2)$, and let $u$ and $u_n$ denote the unique weak solutions to (\ref{E:elliptic}) on $\Omega$ and $\Omega_n$ with $\mathrm{Tr}_\Gamma g$ and $\mathrm{Tr}_{\Gamma_n} g$ in place of $\varphi$, respectively. Then we have 
\[\lim_{n\to\infty} u_n=u\]
KS-strongly with respect to the convergence of Hilbert spaces in Lemma \ref{L:KSconv}.
\end{theorem}

To see Theorem \ref{T:ellipticstability} we can use the following consequence of Lemma \ref{L:Hausdorffconvbd} (i) and the fact that the Whitney extension operator $E_\Gamma$ preserves the Lipschitz constant, \cite[Chapter VI, Theorem 3]{Stein70}.
\begin{lemma}\label{L:traceconv} 
Suppose that Assumptions \ref{A:basicass} and \ref{A:boundaryless} are in force and let $g\in \lip_b(\mathbb{R}^2)$. Then $\lim_{n\to\infty} \mathrm{Tr}_{\Gamma_n} g=\mathrm{Tr}_\Gamma g$ KS-strongly with respect to the convergence of Hilbert spaces in Lemma \ref{L:KSconvbd}.
\end{lemma}

In a similar manner related Cauchy problems can be studied. Let $(\mathcal{L},\mathcal{D}(\mathcal{L}))$ be the infinitesimal generator of $(\mathcal{E},V(\overline{\Omega}))$; it is the superposition of $\mathcal{L}_\Omega$ and $\mathcal{L}_\Gamma^1$. Given $\lambda>0$ and $u_0\in L^2(\overline{\Omega},m)$, we call a function $u:[0,+\infty)\to L^2(\overline{\Omega},m)$ a \emph{solution} to the Cauchy problem 
\begin{equation}\label{E:parabolic}
\begin{cases} \frac{du}{dt}(t)&=(\mathcal{L}-\lambda)u(t), \quad t>0,\\
u(0)&=u_0\end{cases}
\end{equation}
if $u$ is an element of $C^1((0,+\infty),L^2(\overline{\Omega},m))\cap C([0,+\infty),L^2(\overline{\Omega},m))$, $u(t)\in \mathcal{D}(\mathcal{L})$ for all $t>0$ and (\ref{E:parabolic}) holds. Under Assumptions \ref{A:basicass} and \ref{A:Jordan} we can again consider the analogous problem on $\Omega_n$, we use similar notation.

Existence and uniqueness of solutions are clear by semigroup theory.
\begin{lemma}
Let $\lambda>0$. Then for any $u_0\in L^2(\overline{\Omega},m)$ there is a unique solution $u$ to (\ref{E:parabolic}). If Assumptions \ref{A:basicass} and \ref{A:Jordan} are satisfied, then for any large enough $n$ and any $u_{0,n}\in L^2(\overline{\Omega}_n,m_n)$ there is a unique solution $u_n$ to the corresponding problem on $\Omega_n$.
\end{lemma}

Stability now follows from Theorem \ref{T:MoscoWentzell} and \cite[Theorem 2.4]{KuwaeShioya03}.
\begin{theorem}\label{T:parabolicstability}
Suppose that Assumptions \ref{A:basicass}, \ref{A:boundaryless}, \ref{A:Jordan} and \ref{A:uniformeps} are in force. Let
$\lambda>0$ and let $u_0\in L^2(\overline{\Omega},m)$  and $u_{0,n}\in L^2(\overline{\Omega}_n,m_n)$ be such that $\lim_{n\to\infty} u_{0,n}=u_0$ KS-strongly with respect to the convergence of Hilbert spaces in Lemma \ref{L:KSconv}. Let $u$ and $u_n$ denote the unique solutions to (\ref{E:elliptic}) on $\Omega$ and $\Omega_n$ with initial conditions $u_0$ and $u_{0,n}$, respectively. Then for any $t>0$ we have 
\[\lim_{n\to\infty} u_n(t)=u(t)\]
KS-strongly with respect to the convergence of Hilbert spaces in Lemma \ref{L:KSconv}.
\end{theorem}

\appendix

\section{Notions of convergence}\label{S:Notions}

For the convenience of the reader we briefly recall some basic notions from \cite{KuwaeShioya03} in a form adapted to our text. 

Let $H$ and $H_1, H_2, ...$ be separable Hilbert spaces, $\mathcal{C}$ a dense subspace of $H$ and $\Phi_n:\mathcal{C}\to H_n$ linear maps. If 
\begin{equation}\label{E:KSdef}
\lim_{n\to\infty} \|\Phi_n u\|_{H_n}=\|u\|_H,\quad u\in \mathcal{C},
\end{equation}
then we say that the sequence \emph{$(H_n)_n$ converges to $H$ with identification maps $\Phi_n$, $n\geq 1$}. 

Suppose now that this is the case. A sequence $(u_n)_n$ with $u_n\in H_n$ is said to \emph{converge KS-strongly to $u\in H$}
with respect to (\ref{E:KSdef}) if there is a sequence $(\widetilde{u}_m)_m\subset \mathcal{C}$ such that $\lim_{m\to\infty} \widetilde{u}_m=u$ in $H$ and 
\[\lim_{m\to \infty}\limsup_{n\to\infty} \|\Phi_n \widetilde{u}_m-u_n\|_{H_n}=0,\]
see \cite[Definition 2.4]{KuwaeShioya03}. A sequence $(u_n)_n$ with $u_n\in H_n$ is said to \emph{converge KS-weakly to $u\in H$} with respect to (\ref{E:KSdef}) if 
\[\lim_{n\to\infty}\left\langle u_n,v_n\right\rangle_{H_n}=\left\langle u,v\right\rangle_H\]
for any sequence $(v_n)_n$ with $v_n\in H_n$ and any $v\in H$ such $\lim_{n\to\infty} v_n=v$ KS-strongly with respect to (\ref{E:KSdef}), see \cite[Definition 2.4]{KuwaeShioya03}.

Suppose that $Q:H\to [0,+\infty]$ and $Q_n:H_n\to [0,+\infty]$ are quadratic forms. We say that the sequence $(Q_n)_n$ \emph{converges to $Q$ in the KS-generalized Mosco sense} with respect to (\ref{E:KSdef}) if 
\begin{enumerate}
\item[(i)] for any sequence $(u_n)_n$ with $u_n\in H_n$ and any $u\in H$  such that $\lim_{n\to\infty} u_n=u$ KS-weakly with respect to (\ref{E:KSdef}), we have 
\[Q(u)\leq \liminf_{n\to\infty} Q_n(u_n),\]
\item[(ii)] for any $u\in H$ there is a sequence $(u_n)_n$ with $u_n\in H_n$ such that $\lim_{n\to\infty} u_n=u$ KS-strongly and 
\[Q(u)\geq \limsup_{n\to\infty} Q_n(u_n).\]
\end{enumerate}
See \cite[Definition 2.11]{KuwaeShioya03}. As in the case of Mosco convergence on a single Hilbert space, \cite[Definition 2.1.1, Theorem 2.4.1 and Corollary 2.6.1]{Mosco94}, the KS-generalized Mosco convergence is equivalent to the strong convergence of resolvent and semigroup operators, \cite[Theorem 2.4.1]{KuwaeShioya03}.

\vspace{1cm}

{\bf Acknowledgements.} S. C. and M. R. L. have been supported by the Gruppo Nazionale per l'Analisi Matematica, la Probabilit\`a e le loro Applicazioni (GNAMPA) of the Istituto Nazionale di Alta Matematica (INdAM). M. H. has been supported in part by the DFG IRTG 2235 \enquote{Searching for the regular in the irregular:
Analysis of singular and random systems} and by the DFG CRC 1283, \enquote{Taming uncertainty and profiting from randomness
and low regularity in analysis, stochastics and their applications}.


\begin{thebibliography}{99}
\bibitem{Achdou06}
Y. Achdou, Ch. Sabot, N. Tchou, \emph{A multiscale numerical method for Poisson problems in
some ramified domains with fractal boundary}, Multiscale Model. Simul. {\bf 5}(3) (2006), 828--860.
\bibitem{Ahlfors63}
L.V. Ahlfors, \emph{Quasiconformal reflections}, Acta Math. {\bf 109} (1963), 291--301.
\bibitem{AltBeer35}
F. Alt, G. Beer, \emph{Der $n$-Gittersatz in Bogen}, Ergebnisse Math. Koll. {\bf 6}
(1933/34).
\bibitem{ArendtKunkelKunze16}
W. Arendt, S. Kunkel, M. Kunze, \emph{Diffusion with nonlocal boundary conditions}, 
J. Funct. Anal. {\bf 270} (2016), 2483--2507.
\bibitem{ArendtMazzeo07}
W. Arendt, R. Mazzeo, \emph{Spectral properties of the Dirichlet-to-Neumann operator on Lipschitz domains}, Ulmer Seminare {\bf 12} (2007), 23--37.
\bibitem{ArendtterElst11}
W. Arendt, A.F.M. ter Elst, \emph{The Dirichlet-to-Neumann operator on rough domains},
J. Diff. Eq. {\bf 251} (2011), 2100--2124.
\bibitem{Assouad80}
P. Assouad, \emph{Pseudodistances, facteurs et dimension m\'etrique}, S\'eminaire d'Analyse
Harmonique 1979--1980, Publ. Math. Orsay {\bf 80} (1980), 1--33.
\bibitem{Bitsadze69}
A. V. Bitsadze and A. A. Samarski\v{\i}, \emph{On some simple generalizations
of linear elliptic boundary problems}, Sov. Math., Dokl. 10 (1969), 398--400.
\bibitem{BonyCourregePriouret68}
J.-M. Bony, Ph. Courr\`ege, P. Priouret, \emph{Semi-groupes de Feller
sur une vari\'et\'e \`a bord compacte et probl\`emes aux limites int\'egro-
diff\'erentiels du second ordre donnant lieu au principe du maximum},
Ann. Inst. Fourier (Grenoble) {\bf 18} (2) (1968), 369--521.
\bibitem{BH91}
N. Bouleau, F. Hirsch, \emph{Dirichlet Forms and Analysis on Wiener Space},
deGruyter Studies in Math. 14, deGruyter, Berlin, 1991.
\bibitem{BylundGudayol00}
P. Bylund and J. Gudayol, \emph{On the existence of doubling measures with certain
regularity properties}, Proc. Amer. Math. Soc. {\bf 128} (2000), 3317--3327.
\bibitem{CaffarelliSilvestre07}
L. Caffarelli, L. Silvestre, \emph{An extension problem related to the fractional Laplacian},
Comm. Partial Diff. Eq. {\bf 32} (8) (2017), 1245--1260.
\bibitem{CannonMeyer71}
J.R. Cannon, G.H. Meyer, \emph{On a diffusion in a fractured medium}, SIAM J. Appl. Math. {\bf 3}
(1971), 434--448 .
\bibitem{Capitanelli10}
R. Capitanelli, \emph{Robin boundary conditions on scale irregular fractals}, Comm. Pure Appl. Anal. {\bf 9} (5) (2010), 1221--1234.
\bibitem{CapitanelliVivaldi2011}
R. Capitanelli, M.A. Vivaldi, \emph{Insulating layers and Robin Problems on Koch mixtures},
J. Diff. Eq. {\bf 251} (2011), 1332--1353.
\bibitem{CefaloCreoLanciaRodriguez23}
M. Cefalo, S. Creo, M.R. Lancia, J. Rodriguez-Cuadrado, \emph{Fractal mixtures for optimal heat draining},
Chaos, Solitons and Fractals {\bf 173} (2023), 113750.
\bibitem{CefaloCreoLanciaVernole19}
M. Cefalo, S. Creo, M.R. Lancia, P. Vernole, \emph{Nonlocal Venttsel' diffusion in fractal-type domains:
Regularity results and numerical approximation}, Math. Meth. Appl. Sci. {\bf 42} (2019), 4712--4733.
\bibitem{CefaloLancia14}
M. Cefalo, M.R. Lancia, \emph{An optimal mesh generation algorithm for domains with Koch type boundaries}, Math. Comput. Simulation {\bf 106} (2014), 133--162.
\bibitem{ChenKumagai08}
Zh.-Q. Chen, T. Kumagai, \emph{Heat kernel estimates for jump processes of mixed
types on metric measure spaces}, Probab. Theory Related Fields {\bf 140} (1-2) (2008), 277--317.
\bibitem{CreoLanciaNazarov20}
S. Creo, M.R. Lancia, A.I. Nazarov, \emph{Regularity results for nonlocal evolution Venttsel' problems}, Fract. Calc. Appl. Anal. {\bf 23} (5) (2020), 1416--1430.
\bibitem{CreoLanciaNazarovVernole19}
S. Creo, M.R. Lancia, A.I. Nazarov, P. Vernole, \emph{On two-dimensional nonlocal Venttsel' problems in piecewise smooth domains}, Discrete Contin. Dyn. Syst. Series S {\bf 12} (2019), 57--64.
\bibitem{Douglas31}
J. Douglas, \emph{Solution of the problem of Plateau}, Trans. Amer. Math. Soc. {\bf 33} (1) (1931), 263--321.
\bibitem{Dynkin84}
E.M. Dynkin, \emph{Free interpolation by functions with a derivative from $H^1$},
J. Soviet Math. {\bf 27} (1984), 2475--2481.
\bibitem{Engel03}
K.J. Engel, \emph{The Laplacian on $C(\overline{\Omega})$
with generalized Wentzell boundary conditions}, Arch. Math. {\bf 81} (2003), 548--558.
\bibitem{Falconer90}
K. Falconer, \emph{The Geometry of Fractal Sets}, Wiley, Chichester, 1990.
\bibitem{FarkasJacob2001}
W. Farkas, N. Jacob, \emph{Sobolev spaces on non - smooth domains and Dirichlet forms
related to subordinate reﬂecting diffusions}, Math. Nachr. {\bf 224} (2001), 75 -- 104.
\bibitem{Favini02}
A. Favini, G. Ruiz Goldstein, J.A. Goldstein, S. Romanelli, \emph{The heat equation with generalized
Wentzell boundary condition}, J. Evol. Equ. {\bf 2} (2002), 1--19.
\bibitem{Feller51}
W. Feller, \emph{Diffusion processes in genetics}, Proc. Second Berkeley Symp. Math.
Statist. Probab. 1950, Univ. California Press, 1951, pp. 227--246.
\bibitem{Feller52}
W. Feller, \emph{The parabolic differential equations and the associated semi-groups of transformations},
Ann. of Math. (2) {\bf 55} (1952), 468--519.
\bibitem{Feller54}
W. Feller, \emph{Diffusion processes in one dimension}, Trans. Amer. Math. Soc. {\bf 77} (1954), 1--31.
\bibitem{FilocheSapoval00}
M. Filoche, B. Sapoval, \emph{Transfer across random versus deterministic fractal interfaces}, Phys.
Rev. Lett. {\bf 84} (2000), 5776--5779.
\bibitem{FOT94}
M. Fukushima, Y. Oshima and M. Takeda, \emph{Dirichlet forms and symmetric Markov processes},
deGruyter, Berlin, New York, 1994.
\bibitem{GalRuizGoldsteinGoldstein03}
C.G. Gal, G. Ruiz Goldstein, J.A. Goldstein, \emph{Oscillatory boundary conditions for acoustic wave equations},
J. Evol. Equ. {\bf 3} (2003), 623--635.
\bibitem{GalWarma16a}
C.G. Gal, M. Warma, \emph{Transmission problems with nonlocal boundary
conditions and rough dynamic interfaces}, Nonlinearity {\bf 29} (2016), 161.
\bibitem{GalWarma16b}
C.G. Gal, M. Warma, \emph{Long-term behavior of reaction–diffusion equations with nonlocal boundary conditions
on rough domains}, Z. Angew. Math. und Phys. {\bf 67} (2016), 83.
\bibitem{Galakhov01}
E.I. Galakhov, A.L. Skubachevski\v{\i}, \emph{On Feller semigroups generated by elliptic operators with
integro-differential boundary conditions}, J. Diff. Eq. {\bf 176} (2001), 315--355.
\bibitem{Gehring82}
F.W. Gehring, \emph{Characteristic properties of quasidisks}, S\'eminaire de Math\'ematiques Sup\'erieures, vol. 84, Les Presses de l'Universit\'e de Montr\'eal, Montr\'eal, Quebec,
1982.
\bibitem{Gurevich12}
P.L. Gurevich, \emph{Elliptic problems with nonlocal boundary conditions and Feller semigroups},
J. Math. Sci. {\bf 182} (3) (2012), 255-440.
\bibitem{GrothausVosshall17}
M. Grothaus, R. Vosshall, \emph{Stochastic differential equations with
sticky reflection and boundary diffusion},
Electron. J. Probab. {\bf 22} (7) (2017), 1--37.
\bibitem{HaddarJoly01}
H. Haddar, P. Joly, \emph{Effective boundary conditions for thin ferromagnetic layers: The one-dimensional model}, 
SIAM J. Appl. Math. {\bf 61} (4) (2001), 1386--1417.
\bibitem{HerronMeyer12}
D.A. Herron, D. Meyer, \emph{Quasicircles and bounded turning circles
modulo bi-Lipschitz maps}, Rev. Mat. Iberoam. {\bf 28} (2012), 603--630.
\bibitem{Hinz09}
M. Hinz, \emph{Approximation of jump processes on fractals}, Osaka J. Math. {\bf 46} (1) (2009),
141–171.
\bibitem{HLVT18}
M. Hinz, M.R. Lancia, P. Vernole, A. Teplyaev, \emph{Fractal snowflake domain diffusion with boundary and interior drifts}, J. Math. Anal. Appl. {\bf 457}(1) (2018), 672--693.
\bibitem{HinzMeinert20}
M. Hinz, M. Meinert, \emph{On the viscous Burgers equation on metric graphs and fractals}, J. Fractal Geometry {\bf 7} (2) (2020), 137--182.
\bibitem{HinzMeinert22}
M. Hinz, M. Meinert, \emph{Approximation of partial differential equations on compact resistance spaces},  Calc. Var. PDE {\bf 61}, 19 (2022).
\bibitem{HR-PT21}
M. Hinz, A. Rozanova-Pierrat, A. Teplyaev, \emph{Non-Lipschitz uniform domain shape optimization in
linear acoustics}, SIAM J. Control Optim. {\bf 59} (2) (2021), 1007--1032.
\bibitem{HR-PT23}
M. Hinz, A. Rozanova-Pierrat, A. Teplyaev, \emph{Boundary value problems on non-Lipschitz uniform domains: Stability, compactness and the existence of optimal shapes}, Asympt. Anal. {\bf 134} (2023), 25--61.
\bibitem{HinzTeplyaev15}
M. Hinz,  A. Teplyaev, \emph{Closability, regularity, and  approximation  by  graphs  for  separable bilinear forms}, Zap. Nauchn. Sem. S.-Peterburg. Otdel. Mat. Inst. Steklov. (POMI) {\bf 441} (2015), 299--317. Reprint: J. Math. Sci. {\bf 219} (5) (2016), 807--820.
\bibitem{Ikeda61}
N. Ikeda, \emph{On the construction of two-dimensional diffusion processes satisfying Wentzell’s boundary conditions and its application to boundary value problems}, Memoirs of the College of Science, University of Kyoto. Series A: Mathematics, Mem. College Sci. Univ. Kyoto Ser. A Math. {\bf 33} (3) (1961), 367--427.
\bibitem{IkedaWatanabe89}
N. Ikeda, S. Watanabe, \emph{Stochastic differential equations and diffusion processes},  North-Holland, Amsterdam, 
1989.
\bibitem{Ishikawa89}
Y. Ishikawa, \emph{A remark on the existence of a diffusion process
with non-local boundary conditions}, J. Math. Soc. Japan
{\bf 42} (1) (1990), 171--184.
\bibitem{Jones81}
P.W. Jones, \emph{Quasiconformal mappings and extendability of functions in Sobolev spaces}, Acta
Math. {\bf 147} (1981), 71--88.
\bibitem{Jonsson94}
A. Jonsson, \emph{Besov Spaces on Closed Subsets of $\mathbb{R}^n$}, Trans. Amer. Math. Soc. {\bf 341} (1) (1994), 355--370. 
\bibitem{JonssonWallin84}
A. Jonsson and H. Wallin, \emph{Function Spaces on Subsets of $\mathbb{R}^n$}, Math. Reports 2, Part 1, Harwood Acad.
Publ. London, 1984.
\bibitem{Korman83}
P. Korman, \emph{Existence of periodic solutions for a class of nonlinear problems}, Nonlinear
Anal. {\bf 7}, (1983), 873--879.
\bibitem{KuwaeShioya03}
K. Kuwae, T. Shioya, \emph{Convergence of spectral structures: a functional analytic theory
and its applications to spectral geometry}, Comm. Anal. Geom. {\bf 11} (2003), 599--673.
\bibitem{Lancia02}
M.R. Lancia, \emph{A Transmission Problem with a Fractal Interface},
Z. Anal. Anwend. {\bf 21} (1), 113--133.
\bibitem{LanciaCefaloDellAcqua12}
M.R. Lancia, M. Cefalo, G. Dell'Acqua, \emph{Numerical approximation of transmission problems across Koch-type highly conductive layers}, Appl. Math. Comput. {\bf 218} (2012), 5453--5473.
\bibitem{LanciaVelezVernole16}
M.R. Lancia, A. V\'elez-Santiago, P. Vernole, \emph{Quasi-linear Venttsel' problems with nonlocal boundary conditions}, Nonlinear Anal. Real World Appl. {\bf 35} (2017), 265--291.
\bibitem{LanciaVernole14}
M.R. Lancia, P. Vernole, \emph{Venttsel' problems in fractal domains}, J. Evol. Equ. {\bf 14} (3) (2014), 681--712.
\bibitem{Lehto87}
O. Lehto, \emph{Univalent Functions and Teichm\"uller spaces},
Graduate Texts in Mathematics, vol. 109, Springer, New York, 1987.   
\bibitem{LukkainenSaksman98}
J. Luukkainen, E. Saksman, \emph{Every complete doubling metric space carries a
doubling measure}, Proc. Amer. Math. Soc. {\bf 126} (1998), 531--534.
\bibitem{MartioSarvas79}
O. Martio, J. Sarvas, \emph{Injectivity theorems in plane and space}, Ann. Acad. Sci. Fenn. Series A I Math. {\bf 4} (1979), 383--401.
\bibitem{Mattila95}
P. Mattila, \emph{Geometry of Sets and Measures in Euclidean Spaces. Fractals and Rectifiability}, Cambridge studies in adv. math. vol. 44, Cambridge Univ. Press, Cambridge, 1995. 
\bibitem{Mosco94}
U. Mosco, \emph{Composite media and asymptotic Dirichlet forms}, J. Funct. Anal. {\bf 123} (1994),
368--421.
\bibitem{Mosco02}
U. Mosco, \emph{Harnack inequalities on scale irregular Sierpinski gaskets}, In: Nonlinear Problems in
Mathematical Physics and Related Topics II, Int. Math. Ser. vol. 2, edited by M. Birman, S. Hildebrandt, V. Solonnikov, N. Uraltseva, Springer, New York, 2002, pp. 305--328.
\bibitem{Mosco13}
U. Mosco, \emph{Analysis and Numerics of Some Fractal Boundary Value Problems}, In: F. Brezzi, P.  Colli Franzone, U. Gianazza, G. Gilardi (eds) Analysis and Numerics of Partial Differential Equations, Springer INdAM Series, vol 4. Springer, Milano, 2013.
\bibitem{MoscoVivaldi07}
U. Mosco, M.A. Vivaldi, \emph{An example of fractal singular homogenization}, Georgian Math. J.
{\bf 14}(1) (2007), 169--194.
\bibitem{Mugnolo06}
D. Mugnolo, \emph{Abstract wave equations with acoustic boundary conditions},
Math. Nachr. {\bf 279} (3) (2006), 299--318.
\bibitem{MugnoloRomanelli06}
D. Mugnolo, S. Romanelli, \emph{Dirichlet forms for general Wentzell boundary conditions,
analytic semigroups, and cosine operator functions}, Electr. J. Diff. Eq.
{\bf 118} (2006), 1--20.
\bibitem{NicaiseLiMazzucato17}
S. Nicaise, H. Li, A. Mazzucato, \emph{Regularity and a priori error analysis of a Ventcel problem in polyhedral domains}, Math. Methods Appl. Sci. {\bf 40} (2017), 1625--1636.
\bibitem{Nittka11}
R. Nittka, \emph{Regularity of solutions of linear second order elliptic and parabolic boundary value problems on Lipschitz domains}, J. Diff. Eq. {\bf 251} (2011), 860--880.
\bibitem{PhamHuy74}
H. Pham Huy, E. Sanchez Palencia, \emph{Ph\'enom\`enes des transmission \`a travers des couches minces
de conductivit\'e \'elev\'ee}, J. Math. Anal. Appl. {\bf 47} (1974), 284--309.
\bibitem{PostSimmer21}
O.Post, J. Simmer, \emph{Graph-like spaces approximated by discrete graphs and applications},
Math. Nachr. {\bf 294} (11) (2021), 2237--2278.
\bibitem{Rado30}
T. Rad\'o, \emph{On Plateau's problem}, Ann. of Math. (2) {\bf 31} (3) (1930), 457--469.
\bibitem{RieszNagy56}
F. Riesz, B. Sz.-Nagy, \emph{Functional Analysis}, Blackie and Son, London and Glasgow, 1956.
\bibitem{Rohde01}
S. Rohde, \emph{Quasicircles modulo bilipschitz maps}, Rev. Mat. Iberoam. {\bf 17} (2001), 643--659.
\bibitem{Sapoval94}
B. Sapoval, \emph{General formulation of Laplacian transfer across irregular surfaces}, Phys. Rev. Lett.
{\bf 73} (1994), 3314--3316.
\bibitem{SatoUeno65}
K.-I. Sato, T.Ueno, \emph{Multi-dimensional diffusion and the Markov process on the boundary},
J. Math. Kyoto Univ. {\bf 4}(3) (1965), 529--605. 
\bibitem{Schoenberg40}
I.J. Schoenberg, \emph{On metric arcs of vanishing Menger curvature}, Ann. of Math. (2)
{\bf 41} (4) (1940), 715--726.
\bibitem{Skubachevskii97}
A.L. Skubachevskii, \emph{Elliptic Functional Differential Equations and Applications},
Birkh\"auser, Basel, 1997. 
\bibitem{Stein70}
E.M. Stein, \emph{Singular Integrals and Differentiability Properties of Functions}, Princeton Univ. Press, Princeton, 1970.
\bibitem{Taira14}
K. Taira, \emph{Semigroups, Boundary Value Problems and Markov Processes}, 2nd ed., Springer Monographs in Mathematics, Springer, Berlin, 2014.
\bibitem{Ventsell59}
A.D. Venttsel', \emph{On boundary conditions for multidimensional diffusion processes}, Theor. Probability Appl. {\bf 4} (1959), 164--177.
\bibitem{Vaisala88}
J. V\"ais\"al\"a, \emph{Uniform domains}, Tohoku Math. J. {\bf 40} (1988), 101--118.
\bibitem{VolbergKonyagin87}
A. Vol'berg, S. Konyagin, \emph{On measures with the doubling condition}, Izv. Akad.
Nauk SSSR Ser. Mat. {\bf 51} (1987), 666--675.
\bibitem{VogtVoigt03}
H. Vogt, J. Voigt, \emph{Wentzell boundary conditions in the context of Dirichlet forms},
Adv. Diff. Eq. {\bf 8}(7) (2003), 821--842.
\bibitem{Warma13}
M. Warma, \emph{Parabolic and elliptic problems with general Wentzell boundary condition on Lipschitz domains}, Comm. Pure Appl. Anal. {\bf 12} (5) (2013), 1881--1905.
\bibitem{Wu98}
J.-M. Wu, \emph{Hausdorff dimension and doubling measures on metric spaces}, Proc. Amer. Math. Soc. {\bf 126} (1998), 1453--1459.
\end{thebibliography}
\end{document}